\documentclass[10pt,a4paper]{article}
\usepackage{graphicx}
\usepackage{subcaption}
\usepackage{xcolor,comment}
\usepackage[utf8]{inputenc}

\usepackage{hyperref}
\hypersetup{
    colorlinks=true,
    linkcolor=blue,
    citecolor = magenta,
    filecolor=magenta,      
    urlcolor=cyan,
}
\usepackage{siunitx,booktabs}
\hypersetup{breaklinks=true}

\usepackage{arydshln,setspace,amsmath,amsthm,amsfonts,amssymb,mathtools,bm,algorithm,algpseudocode,
longtable,multirow,geometry,mathrsfs,tablefootnote,threeparttable}
\usepackage[affil-it]{authblk}

\usepackage[font=sf, labelfont={sf,bf}, margin=1cm]{caption}
 \usepackage{authblk}

\newtheorem{proposition}{Proposition}

\numberwithin{Theorem}{section}
\numberwithin{Definition}{section}
\numberwithin{Lemma}{section}
\numberwithin{Algorithm}{section}
\numberwithin{equation}{section}

\newtheorem{theorem}{Theorem}[section]

\newtheorem{lemma}[theorem]{Lemma} 
\newtheorem{assumption}{Assumption}

\newtheorem{definition}{Definition}
\newtheorem{remark}{Remark}
\newcommand\scalemath[2]{\scalebox{#1}{\mbox{\ensuremath{\displaystyle #2}}}}

\makeatletter
\def\@cline#1-#2\@nil{%
  \omit
  \@multicnt#1%
  \advance\@multispan\m@ne
  \ifnum\@multicnt=\@ne\@firstofone{&\omit}\fi
  \@multicnt#2%
  \advance\@multicnt-#1%
  \advance\@multispan\@ne
  \leaders\hrule\@height\arrayrulewidth\hfill
  \cr
  \noalign{\nobreak\vskip-\arrayrulewidth}}
\makeatother
\begin{document}
	\title{A Zeroth-order Proximal Stochastic Gradient Method for Weakly Convex Stochastic Optimization\thanks{This version of the paper contains a minor correction of a misprint we identified in the published version as well as updated author affiliations.}}
%\author{Spyridon Pougkakiotis\thanks{\texttt{spyridon.pougkakiotis@yale.edu}}\qquad Dionysios S. Kalogerias\thanks{\texttt{dionysis.kalogerias@yale.edu}}}
%\affil{\textnormal{Yale University}}

\author[1]{Spyridon Pougkakiotis\thanks{\texttt{spyridon.pougkakiotis@kcl.ac.uk}}}
\author[2]{Dionysios S. Kalogerias\thanks{\texttt{dionysis.kalogerias@yale.edu}}}

\affil[1]{\textnormal{King's College London}}
\affil[2]{\textnormal{Yale University}}  % <-- example of a second university

\date{\today}
\maketitle

			%\begin{changemargin}{0.8cm}{0.8cm} 
\begin{abstract}
\par In this paper we analyze a zeroth-order proximal stochastic gradient method suitable for the minimization of weakly convex stochastic optimization problems. We consider nonsmooth and nonlinear stochastic composite problems, for which (sub-)gradient information might be unavailable. The proposed algorithm utilizes the well-known Gaussian smoothing technique, which yields unbiased zeroth-order gradient estimators of a related partially smooth surrogate problem (in which one of the two nonsmooth terms in the original problem's objective is replaced by a smooth approximation). This allows us to employ a standard proximal stochastic gradient scheme for the approximate solution of the surrogate problem, which is determined by a single smoothing parameter, and without the utilization of first-order information. We provide state-of-the-art convergence rates for the proposed zeroth-order method using minimal assumptions. The proposed scheme is numerically compared against alternative zeroth-order methods as well as a stochastic sub-gradient scheme on a standard phase retrieval problem. Further, we showcase the usefulness and effectiveness of our method for the unique setting of automated hyper-parameter tuning. In particular, we focus on automatically tuning the parameters of optimization algorithms by minimizing a novel heuristic model. The proposed approach is tested on a proximal alternating direction method of multipliers for the solution of $\mathcal{L}_1/\mathcal{L}_2$-regularized PDE-constrained optimal control problems, with evident empirical success.
\end{abstract}

\section{Introduction}
\par We are interested in the solution of stochastic weakly convex optimization problems that are not necessarily smooth. Let $(\Omega,\mathscr{F},P)$ be any complete base probability space, and consider a random vector $\xi:\Omega\rightarrow\mathbb{R}^d$.
%where $\omega \in \Omega$, from the probability space into $\mathbb{R}^d$. 
We are interested in stochastic optimization problems of the form
\begin{equation} \label{primal problem} \tag{P}
\underset{x \in \mathbb{R}^n}{\text{min}} \ \phi(x) \coloneqq f(x) + r(x), \qquad  f(x) \coloneqq \mathbb{E}_{\xi}\left[F\left(x,\xi\right)\right],
\end{equation}
\noindent where $F \colon \mathbb{R}^n\times \Xi \rightarrow \mathbb{R}$ is Borel in $\xi$, $f$ is weakly convex, while $r \colon \mathbb{R}^n \rightarrow \overline{\mathbb{R}} \equiv \mathbb{R}\cup \{+\infty\}$ is a proper convex lower semi-continuous function (and hence closed), which is assumed to be proximable (that is, its proximity operator can be computed analytically). 
\par Problem \eqref{primal problem} is very general and appears in a variety of applications arising in signal processing (e.g. \cite{ACHA:EldarMendelson}), optimization (e.g. \cite{MalivertChapter}), engineering (e.g. \cite{Ling_2015}), machine learning (e.g. \cite{NIPS2008:Mairaletal}), and finance (\cite{RISK:RockUry}), to name a few. The reader is referred to \cite[Section 2.1]{SIAMOpt:Davis} and \cite[Section 3.1]{MathProg:DruvPaq} for a plethora of examples. Since neither $f$ nor $r$ are assumed to be smooth, standard stochastic gradient-based schemes are not applicable. In light of this, the authors in \cite{SIAMOpt:Davis} analyzed various model-based stochastic sub-gradient methods (using a standard generalization of the convex subdifferential) for the efficient solution of \eqref{primal problem} and were able to show that convergence is achieved in the sense of near-stationarity of the Moreau envelope of $\phi$ (\cite{MoreauMathFrance}), which serves as a surrogate function with stationary points coinciding with those of \eqref{primal problem}. Given an approximate solution to \eqref{primal problem}, the Moreau envelope offers a way to approximately measure its distance from stationarity in the absence of differentiability. Indeed,  a nearly stationary point for the Moreau envelope is close to a nearly stationary point for the problem under consideration (see \cite[Section 2.2]{SIAMOpt:Davis} or Section \ref{sec: conv analysis}).
\par However, there is a variety of applications in which even sub-gradient information of $f$ (or that of $F(\cdot,\xi)$) might not be available due to the lack of sufficient knowledge about the function (e.g. \cite{SIAMOpt:AudetOrban,Booker1998,SIMAX:Higham}), or such a computation might be prohibitively expensive or noisy (e.g. see \cite{SIAMOpt:Albertoetatl,Kumar_etal,ACM:MezaMartinez}). Thus, several zeroth-order schemes have been developed for the solution of stochastic optimization problems similar to \eqref{primal problem}, requiring only function evaluations of $F(\cdot,\xi)$. Such methods utilize zeroth-order gradient estimates of an appropriate (closely related) surrogate function $F_{\mu}(\cdot,\xi)$ which depends on a smoothing parameter $\mu > 0$.
\par Zeroth-order methods have a long history within the field of optimization (e.g. see the seminal paper on the well-known simultaneous perturbation stochastic approximation (SPSA) \cite{IEETAC:Spall}, the well-known Matyas' method \cite{JOTA:Baba,ARC:Matyas,MathOR:SolisWets}, or the more recent discussion in \cite[Chapter 1]{SIAM:Conn_etal}). However, the relatively recent works on the \emph{Gaussian and uniform smoothing} techniques for convex \cite{IEEE_Inf_Th:Duchi_etal,CompMath:Nesterov_etal} and differentiable non-convex programming \cite{SIAMOpt:GhadimiGuang} have sparked a lot of interest in the literature. Following these developments, the authors in \cite{SIAMOPT:KalogeriasPowellZerothOrder} developed and analyzed a zeroth-order scheme based on the Gaussian smoothing  (see \cite{CompMath:Nesterov_etal}) for the solution of stochastic compositional problems with applications to risk-averse learning, in which $r$ is chosen as an indicator function to a compact convex set. The authors in \cite{FoundCompMath:Bala_etal}, based on the earlier work in \cite{SIAMOpt:GhadimiGuang}, considered (Gaussian smoothing-based) zeroth-order schemes for non-convex Lipschitz smooth stochastic optimization problems, again assuming that $r$ is an indicator function, and focusing on high-dimensionality issues as well as on avoiding saddle-points. We note that the class of non-convex Lipschitz smooth functions is encompassed within the class of weakly convex ones and hence the class of functions appearing in \eqref{primal problem} is strictly wider (see Proposition \ref{Proposition: weak convexity of smooth functions}). In general, there is a plethora of zeroth-order optimization algorithms, and the interested reader is referred to \cite{NEURIPS2018_BalaGhad,SIAM:Conn_etal,arXiv:Dvinskikh_etal,Arxiv:Kozaketal,CompMath:Nesterov_etal,IEETAC:Spall,pmlr-v84-wang18e}, and the references therein.
\par To the best of our knowledge, the only developments on zeroth-order methods for the solution of \eqref{primal problem} can be found in the recent articles given in \cite{COAP:KungurtsevRinaldi,Arxiv:Nazarietal}. The authors in \cite{COAP:KungurtsevRinaldi} utilize a double Gaussian smoothing scheme, which was originally proposed for convex functions in \cite{IEEE_Inf_Th:Duchi_etal}. We argue herein that the use of double smoothing is essentially unnecessary, at least in conjunction with the discussion in \cite{COAP:KungurtsevRinaldi}. In particular, the analysis of the proposed algorithm in \cite{COAP:KungurtsevRinaldi} is substantially more complicated as compared to the analysis provided herein (cf. Section \ref{sec: prox stoch sub-gradient method} and \cite[Section 3]{COAP:KungurtsevRinaldi}), while at the same time offering no advantage in terms of the rate bounds achieved (both here as well as in \cite{COAP:KungurtsevRinaldi} an $\mathcal{O}(\sqrt{n} /T^{1/4})$ rate is shown; cf. Theorem \ref{theorem: convergence general case} and \cite[Theorem 1]{COAP:KungurtsevRinaldi}). Additionally, in \cite{COAP:KungurtsevRinaldi} it is assumed that the iterates produced by the proposed algorithm remain bounded, an assumption that is not required in our analysis. Further, as we show in Section \ref{subsec: Phase retrieval}, the double smoothing approach, except from the fact that it requires the tuning of two smoothing parameters, does not exhibit better convergence behaviour in practice as compared to the proposed method herein. On the other hand, the authors in \cite{Arxiv:Nazarietal} present an adaptive zeroth-order method for problems of the form of \eqref{primal problem} using a uniform smoothing scheme. However, the analysis in the aforementioned paper yields a worse dependence on the problem dimensions $n$ than that obtain herein, while at the same time requires certain additional restrictive assumptions (in particular, an $\mathcal{O}(n/T^{1/4})$ convergence rate is shown, cf. Theorem \ref{theorem: convergence general case} and \cite[Corollary 19]{Arxiv:Nazarietal}, and the authors assume that the iterates lie in a compact set and that the function $F(\cdot,\xi)$ is Lipschitz continuous with a constant that does not depend on $\xi$; neither of these is assumed in our analysis). 
\par Instead, in this paper we develop and analyze a zeroth-order proximal stochastic gradient method for the solution of \eqref{primal problem}, utilizing standard (single) Gaussian smoothing (see \cite{CompMath:Nesterov_etal}). Following the developments in \cite{SIAMOpt:Davis}, we analyze the algorithm and show that it obtains an $\epsilon$-stationary solution to the Moreau envelope of an appropriate \emph{surrogate problem} in at most $\mathcal{O}(n^2\epsilon^{-4})$ iterations; a state-of-the-art bound of the same order as the bound achieved by sub-gradient schemes (see \cite{SIAMOpt:Davis}), up to a constant term depending on the dimension of $x$. This rate matches the one shown in \cite{COAP:KungurtsevRinaldi} for the double Gaussian smoothing scheme, however, the proposed analysis is significantly easier, and does not assume boundedness of the iterates, which is required for the analysis in \cite{COAP:KungurtsevRinaldi}. Additionally, given any near-stationary solution to the surrogate problem for which the convergence analysis is performed, we show that it is a near-stationary solution for the Moreau envelope of the original problem. Such a connection is easy to establish when $r$ is an indicator function (e.g. see \cite{SIAMOPT:KalogeriasPowellZerothOrder}), however not so obvious for general closed convex functions $r$ that are studied here. Indeed, this was not considered in \cite{COAP:KungurtsevRinaldi}. A rate directly related to the Moreau envelope of the original problem is given in the analysis in \cite{Arxiv:Nazarietal} (where a uniform smoothing scheme is studied), however, the analysis in the aforementioned work utilizes additional restrictive assumptions to achieve this (as previously mentioned, boundedness of the problem's domain and Lipschitz continuity of $F(\cdot,\xi)$ with a uniform Lipschitz constant for all $\xi$), while an $\mathcal{O}(n^4\epsilon^{-4})$ rate is shown (i.e. a significantly worse dependence on the problem dimensions $n$).
\par In order to empirically stress the viability and usefulness of the proposed approach, we consider two problems. Initially, we test our method on several phase-retrieval instances taken from \cite{SIAMOpt:Davis}, and compare its numerical behaviour against a sub-gradient model-based scheme developed in \cite{SIAMOpt:Davis}, as well zeroth-order stochastic gradient schemes based on the double Gaussian smoothing, the uniform smoothing, and the SPSA. The observed numerical behaviour confirms the theory, in that the proposed zeroth-order method converges consistently at a rate that is slower only by a constant factor than that exhibited by the sub-gradient scheme, while it is competitive against all other zeroth-order schemes. Subsequently, we showcase that the practical performance of the proposed algorithm is seemingly identical to that achieved by the double smoothing zeroth-order scheme analyzed in \cite{COAP:KungurtsevRinaldi}, even if the two smoothing parameters of the latter are tuned. 
\par Next, we consider a very important application of zeroth-order (or in general derivative-free) optimization; that is hyper-parameter tuning. This is a very old problem (traditionally appearing in the industry, e.g. see \cite{Booker1998}, and often solved by hand via exhausting or heuristic random search schemes) that has seen a surge in importance in light of the recent developments in artificial intelligence and machine learning. There is a wide literature on this subject, which can only briefly be mentioned here. The most common approaches are based on Bayesian optimization techniques (e.g. see \cite{NIPS2011_Bergstra_etal,JMLR:v13:bergstra12a,Feurer2019}), although derivative-free schemes have also been considered (e.g. see \cite{SIAMOpt:AudetOrban}). In certain special cases, application specific automated tuning strategies have also been investigated (e.g. see \cite{ETNA:Calvettietal,BIT:Fenuetal,Pragliola_etal}). Given the importance of hyper-parameter tuning, there have been developed several heuristic software packages for this purpose, such as the Nevergrad toolkit (see \cite{software:nevergrad}). In this paper, we consider the problem of tuning the parameters of optimization algorithms. To that end, we derive a novel heuristic model, the minimization of which yields the hyper-parameters that minimize the residual reduction of an optimization algorithm that depends on them, after a fixed given number of iterations, for an arbitrary class of optimization problems (assumed to follow an unknown distribution from which we can sample). Focusing on a proximal alternating direction method of multipliers (pADMM), we tune its penalty parameter for two problem classes; the optimal control of the Poisson equation as well as the optimal control of the convection-diffusion equation. In both cases we numerically verify the efficient performance of the pADMM with the ``learned" hyper-parameter when considering out-of-sample instances. The MATLAB implementation is provided.
\paragraph{Notation}
We denote by $\langle \cdot, \cdot \rangle$ the inner product in $\mathbb{R}^n$, and given a vector $x \in \mathbb{R}^n$, $\|x\|_2$ denotes the induced Euclidean norm. Given a complete probability space $(\Omega,\mathscr{F},P)$, where $\mathscr{F}$ is a sigma algebra and $P$ is a probability measure, we denote by $\mathcal{L}_{p}(\Omega,\mathscr{F},P;\mathbb{R})$, for some $p \in [1,+\infty)$, the space of all $\mathscr{F}$-measurable functions $\varphi \colon \Omega \rightarrow \mathbb{R}$ such that $\left(\int_{\Omega} \left|\varphi(\omega) \right|^p dP(\omega)\right)^{1/p} < +\infty$. Given a random vector $Z \colon \Omega \rightarrow \mathbb{R}^d$, and a random function $\varphi \colon \mathbb{R}^d \rightarrow \overline{\mathbb{R}}$, we denote the expected value as $E_Z[\varphi(Z)] = \int_{\Omega} \varphi\left(Z(\omega)\right)dP(\omega)$, where the subscript is employed to stress that the expectation is taken with respect to the random variable $Z$. Finally, given a function $\varphi \colon \mathbb{R}^n \rightarrow \mathbb{R}^m$, we say that $\varphi$ is Lipschitz continuous on a set $X \subset \mathbb{R}^n$ if there is a constant $c \geq 0$ such that $\|\varphi(x_1) - \varphi(x_2)\|_2 \leq c \|x_1 - x_2\|_2$, for all $x_1,\ x_2 \in X$. If $\varphi$ is Lipschitz continuous on a neighbourhood of every point of $X$ (potentially with different Lipschitz constants), then it is said that $\varphi$ is locally Lipschitz continuous on $X$.
\paragraph{Structure of the article}
The rest of this paper is organized as follows. In Section \ref{sec: Preliminaries} we introduce some notation as well as preliminary notions of significant importance for the developments in this paper. In Section \ref{sec: prox stoch sub-gradient method} we derive and analyze the proposed zeroth-order proximal stochastic gradient method for the solution of \eqref{primal problem}. In Section \ref{sec: numerical results} we present some numerical results, and in Section \ref{sec: Conclusions} we derive our conclusions.

\section{Preliminaries} \label{sec: Preliminaries}
\par In this section, we introduce some preliminary notions that will be used throughout this paper. In particular, we first discuss certain core properties of stochastic weakly convex functions of the form of $f$. Subsequently, we introduce the Gaussian smoothing (e.g. see \cite{SIAMOPT:KalogeriasPowellZerothOrder,CompMath:Nesterov_etal}), which provides a smooth surrogate for $f$ in \eqref{primal problem}. In turn, this can be used to obtain zeroth-order optimization schemes; such methods are only allowed to access a zeroth-order oracle (i.e. only sample-function evaluations are available). In turn, the Gaussian smoothing guides us in the choice of minimal assumptions on the stochastic part of the objective function in \eqref{primal problem}. Finally, we introduce the proximity operator, as well as certain core properties of it. These notions will then be used to derive a zeroth-order proximal stochastic gradient method in Section \ref{sec: prox stoch sub-gradient method}.
\subsection{Stochastic weakly convex functions}
\par Let us briefly discuss some core properties of the well-studied class of weakly convex functions. For a detailed study on the properties of these functions (and of related sets), the reader is referred to \cite{MathOR:Vial}, and the references therein. Below we define the class of weakly convex functions for completeness.
\begin{definition}
Let $f \colon \mathbb{R}^n \mapsto \mathbb{R}$. It is said to be $\rho$-weakly convex, for some $\rho > 0$, if for any $x_1,\ x_2 \in \mathbb{R}^n$, and any $\lambda \in [0,1]$, it holds that
\[f\left(\lambda x_1 + (1-\lambda)x_2\right) \leq \lambda f(x_1) + (1-\lambda)f(x_2) + \frac{\lambda(1-\lambda)\rho}{2}\left\|x_1-x_2\right\|_2^2.\]
\end{definition}
\noindent In what follows, we make use of a standard generalization of the well-known convex subdifferential (which consists of all global affine under-estimators of a convex function at a given point). Specifically, we consider the subdifferential that consists of all global concave quadratic under-estimators (see \cite[Section 2.2]{SIAMOpt:Davis}). In particular, given a locally Lipschitz continuous function $f \colon \mathbb{R}^n \mapsto \overline{\mathbb{R}}$, and some $x \in \textnormal{dom}(f)$, we define the generalized subdifferential $\partial f(x)$ as the set of all vectors $v \in \mathbb{R}^n$ satisfying
\[ f(y) \geq f(x) + \langle v,y-x\rangle + o\left(\|y-x\|_2\right),\qquad \textnormal{as }y \rightarrow x, \]
\noindent and set $\partial f(x) = \emptyset$ for any $x \notin \textnormal{dom}(f)$. A more general definition, based on the Clarke generalized directional derivative (see \cite{AMS:Clarke}), can be found in \cite[Section 1]{MathOR:Vial}.
\noindent We note that the mapping $x \mapsto \partial f(x)$ of a weakly convex function $f$ inherits many properties of the subgradient mapping of a convex function (see \cite[Section 4]{MathOR:Vial}), and reduces to the standard convex subdifferential if $f$ is a convex function. In the following proposition we state some important properties holding for weakly convex functions.
\begin{proposition} \label{Proposition: weak convexity properties}
Any $\rho$-weakly convex function $f \colon \mathbb{R}^n \mapsto \mathbb{R}$ is locally Lipschitz continuous and regular in the sense of Clarke, and thus directionally differentiable. Furthermore, it is bounded below, and there exists $z \in \mathbb{R}^n$ such that
\[f(x_2) \geq f(x_1) + \left\langle z,x_2-x_1\right\rangle - \frac{\rho}{2}\left\|x_2-x_1 \right\|_2^2.\]
\noindent Moreover, the latter holds for any $z \in \partial f(x_1)$. Finally, the map $x \mapsto f(x) + \frac{\rho}{2}\|x\|_2^2$ is convex and 
\[ \langle z_1-z_2,x_1-x_2 \rangle \geq -\rho \|x_1-x_2\|_2^2,\]
\noindent for all $x_1, x_2 \in \mathbb{R}^n$, $z_1 \in \partial f(x_1)$, and $z_2 \in \partial f(x_2)$.
\end{proposition} 
\begin{proof}
The proof can be found in \cite[Propositions 4.4, 4.5, and 4.8]{MathOR:Vial}.
\end{proof}
\begin{proposition} \label{Proposition: weak convexity of smooth functions}
Any continuously differentiable function $f \colon \mathbb{R}^n \rightarrow \mathbb{R}$, with globally $\rho$-Lipschitz gradient, where $\rho > 0$, is $\rho$-weakly convex.
\end{proposition}
\begin{proof}
The proof follows trivially from Proposition \ref{Proposition: weak convexity properties}, see \cite[Proposition 4.12]{MathOR:Vial}.
\end{proof}
\subsection{Gaussian smoothing}
\par Let us introduce the notion Gaussian smoothing. To that end, we follow the notation adopted in \cite{SIAMOPT:KalogeriasPowellZerothOrder}. Let $f \colon \mathbb{R}^n \rightarrow \mathbb{R}$ be a Borel function, and $U \sim \mathcal{N}\left(0_n, I_n\right)$ a normal random vector, where $I_n$ is the identity matrix of size $n$. Given a non-negative smoothing parameter $\mu \geq 0$, the Gaussian smoothing of $f$ is defined as
\[f_{\mu}(\cdot) \coloneqq \mathbb{E}_U\left[f\left(\left(\cdot\right) + \mu U \right)\right], \]
\noindent assuming that the expectation is well-defined and finite for all $x \in \mathbb{R}^n$. The precise conditions on $F(x,\xi)$ (in \eqref{primal problem}) for this to hold will be given later in this section. Let $\mathcal{N} \colon \mathbb{R}^n \rightarrow \mathbb{R}$, with a slight abuse of notation, be the standard Gaussian density in $\mathbb{R}^n$, that is the mapping $x \mapsto \frac{1}{(2\pi)^{n/2}}e^{-\frac{1}{2}x^\top x}$. Then, we can observe that:
\[ f_{\mu}(x) = \int f(x + \mu u) \mathcal{N}\left(u\right) du = \mu^{-n}\int f(v) \mathcal{N}\left(\frac{v-x}{\mu}\right) dv,\]
\noindent where the second equality holds via introducing an integration variable $v = x + \mu u$. The second characterization yields the following expressions for the gradient of $f_{\mu}$ (assuming it exists):
\begin{equation*}
\begin{split}
\nabla f_{\mu}(x) =&\  \mu^{-(n+2)} \int f(v) \mathcal{N}\left(\frac{v-x}{\mu}\right)(v-x)dv \\
=&\ \mu^{-1} \int f(x+\mu u) \mathcal{N}\left(u\right)u du\\
=&\ \mathbb{E}_U\left[\frac{f\left(x+\mu U\right) - f(x)}{\mu}U\right] \\
=&\ \mathbb{E}_U\left[\frac{f\left(x + \mu U\right) - f\left(x-\mu U\right)}{2\mu} U \right],
\end{split}
\end{equation*}
\noindent where $U \sim \mathcal{N}\left(0_n,I_n\right)$. The second equality follows from a change of variables, the third from the properties of the standard Gaussian, while the last one can be trivially shown by direct computation (e.g. see \cite{CompMath:Nesterov_etal}). 
\par In what follows, we impose certain assumptions on the function $F$ given (implicitly) in \eqref{primal problem}, in order to guarantee that its Gaussian smoothing is well-defined and satisfies several properties of interest. 
\begin{assumption} \label{Assumption: F assumptions}
Let $F \colon \mathbb{R}^n\times \Xi \rightarrow \mathbb{R}$ satisfy the following properties:
\begin{itemize}
\item[\textnormal{(\textbf{C1})}] $F(x,\cdot) \in \mathcal{L}_2\left(\Omega,\mathscr{F},P; \mathbb{R}\right)$, and is Borel for any $x \in \mathbb{R}^n$.
\item[\textnormal{(\textbf{C2})}] The function $f(x) = \mathbb{E}_{\xi}[F\left(x,\xi\right)]$ is $\rho$-weakly convex for some $\rho \geq 0$.
\item[\textnormal{(\textbf{C3})}] There exists a positive random variable $C(\xi)$ such that $\sqrt{\mathbb{E}_{\xi}\left[C(\xi)^2\right]} < \infty$, and for all $x_1,\ x_2 \in \mathbb{R}^n$, and a.e. $\xi \in \Xi$, the following holds:
\[\left\lvert F(x_1,\xi) - F(x_2,\xi)\right\rvert \leq C(\xi)\|x_1-x_2\|_2. \]
\end{itemize}
\end{assumption}
\begin{remark}
\par In view of \textnormal{(\textbf{C1})} in Assumption \textnormal{\ref{Assumption: F assumptions}}, we can infer that $f$ is well-defined and finite for any $x$. In fact, this can be shown with a weaker condition in place of \textnormal{(\textbf{C1})}, that is, if we were to assume that $F(x,\cdot) \in \mathcal{L}_1\left(\Omega,\mathscr{F},P;\mathbb{R}\right)$ for any $x \in \mathbb{R}^n$. The stronger assumption will be utilized in Lemma \textnormal{\ref{lemma:Gaussian smoothing properties}}. Furthermore, from \textnormal{\cite[Theorem 7.44]{SIAM:Shapiro_etal}}, under \textnormal{(\textbf{C1})} and \textnormal{(\textbf{C3})}, it follows that there exists a constant $L_{f,0} > 0$, such that $f$ is $L_{f,0}$-Lipschitz continuous on $\mathbb{R}^n$. Again, this holds even if we weaken assumption \textnormal{(\textbf{C3})}, and only require that $\mathbb{E}_{\xi}\left[C(\xi)\right] < \infty$, however, the stronger form of this assumption is utilized in Lemma \textnormal{\ref{lemma:Gaussian smoothing properties}}. 
\end{remark}
\par Under Assumption \ref{Assumption: F assumptions}, we will provide certain properties of the surrogate function $f_{\mu}$, as presented in \cite{CompMath:Nesterov_etal}.
\begin{lemma} \label{lemma:Gaussian smoothing properties}
Let Assumption \textnormal{\ref{Assumption: F assumptions}} hold. Then, $f_{\mu}$ is $\rho$-weakly convex, and there exists a constant $L_{f_{\mu},0} \leq L_{f,0}$ such that $f_{\mu}$ is $L_{f_{\mu},0}$-Lipschitz continuous on $\mathbb{R}^n$. Additionally, for any $\mu \geq 0$, we obtain
\begin{equation} \label{eqn: approximation error of surrogate}
\left\lvert f_{\mu}(x) - f(x)\right\rvert \leq \mu L_{f,0} n^{\frac{1}{2}}, \qquad \text{for any}\ x \in \mathbb{R}^n,
\end{equation}
\noindent while for any $\mu > 0$, $f_{\mu}$ is Lipschitz continuously differentiable with 
\begin{equation} \label{eqn: gradient of the surrogate of expectation}
\nabla f_{\mu}(x) = \mathbb{E}_U\left[\frac{f\left(x+\mu U\right) - f(x) }{\mu} U \right] = \mathbb{E}_{U, \xi}\left[\frac{F\left(x+\mu U,\xi\right)-F(x,\xi)}{\mu} U\right],
\end{equation}
where $U,\ \xi$ are statistically independent. Additionally, we have that
\begin{equation} \label{eqn: bound on the expected gradient}
\mathbb{E}_{U,\xi}\left[\left\|\frac{F\left(x+ \mu U,\xi\right)- F(x,\xi)}{\mu} U \right\|_2^2\right] \leq (n^2 + 2n) L_{f,0}^2.
\end{equation}
\end{lemma}
\begin{proof}
Weak convexity of the surrogate can be obtained by \cite[Lemma 5.2]{SIAMOPT:KalogeriasPowellZerothOrder}. For a proof of \eqref{eqn: approximation error of surrogate}, as well as the first equality of \eqref{eqn: gradient of the surrogate of expectation}, the reader is referred to \cite[Appendix]{CompMath:Nesterov_etal}. The second equality in \eqref{eqn: gradient of the surrogate of expectation}, in light of (\textbf{C3}) of Assumption \ref{Assumption: F assumptions}, follows by Fubini's theorem (we should note that with a slight abuse of notation, the second expectation in \eqref{eqn: gradient of the surrogate of expectation} is taken with respect to the product measure of the two corresponding random vectors $U$ and $\xi$).  Following the developments in \cite[Lemma 5.4]{SIAMOPT:KalogeriasPowellZerothOrder}, we show \eqref{eqn: bound on the expected gradient}. In particular, we have
\begin{equation*}
\begin{split}
\scalemath{0.9}{\mathbb{E}_{U,\xi}\left[\left\|\frac{F\left(x+ \mu U,\xi\right)- F(x,\xi)}{\mu} U \right\|_2^2\right]} & = \frac{1}{\mu^2}\mathbb{E}_{U,\xi}\left[\left\lvert F\left(x+ \mu U,\xi\right)- F(x,\xi)\right\rvert^2 \left\|U \right\|_2^2\right] \\
& = \frac{1}{\mu^2}\mathbb{E}_U\left[ \mathbb{E}_{\xi}\left[ \left\lvert F\left(x+ \mu U,\xi\right)- F(x,\xi)\right\rvert^2 \left\|U \right\|_2^2 \middle\vert U\right] \right]\\
& = \frac{1}{\mu^2}\mathbb{E}_U\left[ \mathbb{E}_{\xi}\left[ \left\lvert F\left(x+ \mu U,\xi\right)- F(x,\xi)\right\rvert^2  \middle\vert U\right] \left\|U \right\|_2^2\right]\\
& \leq L_{f,0}^2 \mathbb{E}_U\left[ \|U\|_2^4 \right] = (n^2 + 2n)L_{f,0}^2,
\end{split}
\end{equation*}
\noindent where in the second equality we used the tower property, while in the last line we employed (\textbf{C3}), and evaluated the 4-th moment of the $\chi$-distribution.
\end{proof}
\subsection{Proximal point and the Moreau envelope}
\par At this point, we briefly discuss certain well-known notions for completeness. More specifically, given a closed function $p \colon \mathbb{R}^n \rightarrow \overline{\mathbb{R}}$, and a positive penalty $\lambda > 0$, we define the proximal point
\[ \textbf{prox}_{\lambda p}(u) \coloneqq \arg\min_x \left\{p(x) + \frac{1}{2\lambda} \|u-x\|_2^2 \right\},\]
\noindent as well as the corresponding Moreau envelope
\[ p^{\lambda}(u) \coloneqq \min_x \left\{p(x) + \frac{1}{2\lambda}\|x-u\|_2^2\right\} = p\left(\textbf{prox}_{\lambda p}(u)\right) + \frac{1}{2\lambda}\left\|\textbf{prox}_{\lambda p}(u) - u \right\|_2^2.\]
\noindent We can show (e.g. see \cite{SIAMOpt:Davis,MoreauMathFrance}) that if $p$ is $\rho$-weakly convex, for some $\rho > 0$, then $p_{\lambda}$ is continuously differentiable for any $\lambda \in \left(0,\rho^{-1}\right)$, with
\[\nabla p^{\lambda}(u) = \lambda^{-1}\left(u -  \textbf{prox}_{\lambda p}(u)\right).\]
\par The Moreau envelope has been used as a smooth penalty function for line-search in Newton-like methods (e.g. see \cite{IEEE_DC:Patrinos_etal}). More recently, it was noted in \cite[Section 2.2]{SIAMOpt:Davis} that the norm of its gradient (that is $\|\nabla p^{\lambda}(u)\|_2$) can serve as a near-stationarity measure for nonsmooth optimization. The latter approach is adopted in this paper, and thus, we will later on derive a convergence analysis of the proposed zeroth-order proximal stochastic gradient method based on the magnitude of the gradient of an appropriate Moreau envelope.
\iffalse
\par Below we provide a useful property of the Moreau envelope of a $\rho$-weakly convex function, that will later be utilized.
\begin{lemma} \label{Lemma: weak convexity of Moreau envelope}
Let $f \colon \mathbb{R}^n \mapsto \mathbb{R}$ be a $\rho$-weakly convex function, with $\rho > 0$, and denote by $f^{\lambda}$ its corresponding Moreau envelope, with $\lambda \in (0,\rho^{-1})$. Then, $f^{\lambda}$ is $2 \lambda^{-1}$-weakly convex.
\end{lemma}
\begin{proof}
\par Since $\lambda \in (0,\rho^{-1})$, we have that
\[\nabla f^{\lambda}(x) = \lambda^{-1}\left(x - \textbf{prox}_{\lambda f}(x)\right).\]
\noindent Thus, for and $x_1,\ x_2 \in \mathbb{R}^n$, we obtain
\begin{equation*}
\begin{split}
 \left\| \nabla f^{\lambda}(x_1) - \nabla f^{\lambda}(x_2)\right\|_2 = &\ \left\| \lambda^{-1}\left(x_1 - \textbf{prox}_{\lambda f}(x_1) -x_2 + \textbf{prox}_{\lambda f}(x_2)\right)\right\|_2\\
 \leq &\ 2 \lambda^{-1}\|x_1 - x_2\|_2,
 \end{split}
 \end{equation*}
 \noindent where we used the triangular inequality as well as the non-expansiveness of the proximal operator. This shows that $f^{\lambda}$ has globally Lipschitz gradient with Lipschitz constant equal to $2\lambda^{-1}$. The result follows immediately by applying Proposition \ref{Proposition: weak convexity of smooth functions}.
\end{proof}
\fi
\section{A zeroth-order proximal stochastic gradient method} \label{sec: prox stoch sub-gradient method}
\par In this section we derive a zeroth-order proximal stochastic gradient method suitable for the solution of problems of the form of \eqref{primal problem}. Let us employ the following assumption:
\begin{assumption} \label{Assumption: final assumption}
Let $F(x,\xi)$ be defined as in \textnormal{\eqref{primal problem}} satisfying  Assumption \textnormal{\ref{Assumption: F assumptions}}. Additionally,  we assume that $r$ is a proper (i.e. $\textnormal{dom}(r) \neq \emptyset$) closed convex function (and thus lower semi-continuous), and proximable (that is, its proximity operator can be evaluated analytically). Finally, we can generate two statistically independent random sequences $\{U_i\}_{i = 0}^{\infty},\ \{\xi_i\}_{i=0}^{\infty}$, such that each $U_i \sim \mathcal{N}\left(0_n,I_n\right)$ and $\xi_i$ is i.i.d., respectively.
\end{assumption}
\par In light of Assumption \ref{Assumption: final assumption}, and by utilizing Lemma \ref{lemma:Gaussian smoothing properties}, we can quantify the quality of the approximation of $\phi(x)$ by $\phi_{\mu}(x) \coloneqq f_{\mu}(x) + r(x)$, for any $x \in \mathbb{R}^n$. Additionally, we know that $f_{\mu}$ is smooth, even if $f$ is not. Thus, we can derive an optimization algorithm for the minimization of $\phi_{\mu}$ (which can utilize stochastic gradient approximations for the smooth function $f_{\mu}$), and then retrieve an approximate solution to the original problem, where the approximation accuracy can be directly controlled by the smoothing parameter $\mu$. Thus, we analyze a zeroth-order stochastic optimization method for the solution of the following surrogate problem
\begin{equation} \label{primal surrogate problem}
\min_x\ \phi_{\mu}(x) \coloneqq f_{\mu}(x) + r(x),  \tag{\mbox{$\textnormal{P}_{\mu}$}}
\end{equation}
\noindent where $f_{\mu}(x) = \mathbb{E}_{U}\left[f\left( x + \mu U\right)\right]$, $\mu > 0$, and $f$, $r$ are as in \eqref{primal problem}. The method is summarized in Algorithm \ref{Algorithm: Z-ProxSG}.

\renewcommand{\thealgorithm}{Z-ProxSG}

\begin{algorithm}[!ht]
\caption{Zeroth-Order Proximal Stochastic Gradient}
    \label{Algorithm: Z-ProxSG}

\begin{algorithmic}
\State \textbf{Input:}  $x_0 \in \textnormal{dom}(r)$, a sequence $\{\alpha_t\}_{t \geq 0} \subset \mathbb{R}_+$, $\mu > 0$, and $T > 0$.
\For {($t = 0,1,2,\ldots, T$)}
\State Sample $\xi_t$, $U_t \sim \mathcal{N}\left(0_n,I_n\right)$, and set
\[x_{t+1} = \textbf{prox}_{\alpha_t r}\left(x_t - \alpha_t G\left(x_t,U_t,\xi_t\right)\right), \]
\State where $G\left(x_t,U_t,\xi_t\right) \coloneqq \mu^{-1}\left(F\left(x_t + \mu U_t,\xi_t\right) - F(x_t,\xi_t)\right) U_t$.
\EndFor
\State Sample $t^* \in \{0,\ldots,T\}$ according to $\mathbb{P}(t^* = t) = \frac{\alpha_t}{\sum_{i = 0}^T\alpha_i}$.
\State \Return $x_{t^*}$.
\end{algorithmic}
\end{algorithm}
\subsection{Convergence analysis} \label{sec: conv analysis}
\par In what follows, we derive the convergence analysis for Algorithm \ref{Algorithm: Z-ProxSG}. We obtain the rate of the proposed algorithm for finding a nearly-stationary solution to the surrogate problem \eqref{primal surrogate problem} (see Theorem \ref{theorem: convergence general case}), and then by utilizing Lemma \ref{lemma:Gaussian smoothing properties}, we argue that a nearly-stationary solution of the surrogate problem is nearly-stationary for the Moreau envelope of problem \eqref{primal problem} (see Theorem \ref{theorem: bound on Moreau envelope gradient of original problem}). The analysis follows closely the developments in \cite[Section 3.2]{SIAMOpt:Davis}. 
\par Let us first introduce some notation. Set $\bar{\rho} \in (\rho,2\rho]$, where $\rho$ is the weak-convexity constant of $f(\cdot)$. We define $\hat{x}_{t} \coloneqq \textnormal{\textbf{prox}}_{\bar{\rho}^{-1}\phi_{\mu}}(x_t)$, and $\delta_t \coloneqq 1-\alpha_t \bar{\rho}$. The auxiliary point $\hat{x}_t$ is the ``optimal" proximal step at iteration $t$. In Lemma \ref{lemma: descent property of iterates}, we show how far is the new iterate of Algorithm \ref{Algorithm: Z-ProxSG} (in expectation) from this ``optimal" proximal step. In turn, this bound is then utilized in Theorem \ref{theorem: convergence general case} to show convergence in terms of reduction of the gradient norm of the surrogate Moreau envelope. The following lemma introduces a useful property of this auxiliary point.
\begin{lemma} \label{lemma: x_hat prox representation w.r.t. r}
For any $t \geq 0$, and any iterate $x_t$ of Algorithm \textnormal{\ref{Algorithm: Z-ProxSG}}, we obtain
\[\hat{x}_{t} = \textnormal{\textbf{prox}}_{\alpha_t r}\left( \alpha_t\bar{\rho}x_t - \alpha_t \nabla f_{\mu}(x_t) + \delta_t \hat{x}_t \right). \]
\end{lemma}
\begin{proof}
\noindent See Appendix \ref{Appendix: proof of lemma x_hat prox representation w.r.t. r}.
\end{proof}
\par Following \cite{SIAMOpt:Davis}, we derive a descent property for the iterates.
\begin{lemma} \label{lemma: descent property of iterates}
Let Assumption \textnormal{\ref{Assumption: final assumption}} hold, set $\bar{\rho} \in (\rho,2\rho]$, and choose $\alpha_t \in \left(0,1/\bar{\rho}\right]$, for any $t \geq 0$. Then, the following inequality holds:
\[\mathbb{E}_{U,\xi}^t \left[\|x_{t+1} - \hat{x}_t\|_2^2\right] \leq \|x_t - \hat{x}_t\|_2^2 + 4(n^2+2n)\alpha_t^2 L_{f,0}^2 - 2\alpha_t (\bar{\rho}-\rho)\|x_t - \hat{x}_t\|_2^2,\] 
\noindent where $\mathbb{E}_{U,\xi}^t \left[\cdot\right] \equiv \mathbb{E}_{U,\xi} \left[\cdot \middle\vert U_{t-1}, \xi_{t-1},\ldots,U_0, \xi_0\right].$
\end{lemma}
\begin{proof}
\par We have
\begin{equation*}
\begin{split}
&\scalemath{1}{\mathbb{E}_{U,\xi}^t \left[\|x_{t+1} - \hat{x}_t\|_2^2\right]  } \\ &\quad = \scalemath{0.95}{\mathbb{E}_{U,\xi}^t \left[ \left\|\textbf{prox}_{\alpha_t r}\left( x_t - \alpha_t G\left(x_t,U_t,\xi_t\right)\right) - \textbf{prox}_{\alpha_t r} \left(\alpha_t \bar{\rho}x_t - \alpha_t \nabla f_{\mu}(\hat{x}_t) + \delta_t \hat{x}_t \right) \right\|_2^2\right] }\\
&\quad \leq \  \mathbb{E}_{U,\xi}^t \left[\left\|\left( x_t - \alpha_t G\left(x_t,U_t,\xi_t\right)\right)- \left(\alpha_t \bar{\rho}x_t - \alpha_t \nabla f_{\mu}(\hat{x}_t) + \delta_t \hat{x}_t \right)\right\|_2^2\right] \\
&\quad =\ \delta_t^2\|x_t - \hat{x}_t\|_2^2 - 2\delta_t \alpha_t \mathbb{E}_{U,\xi}^t \left[ \left\langle x_t - \hat{x}_t,G\left(x_t,U_t, \xi_t\right) - \nabla f_{\mu}(\hat{x}_t)\right\rangle\right]\\
&\qquad + \alpha_t^2 \mathbb{E}_{U,\xi}^t \left[ \|G\left(x_t,U_t, \xi_t\right) - \nabla f_{\mu}(\hat{x}_t)\|_2^2\right] \\
&\quad\leq \ \delta_t^2\|x_t - \hat{x}_t\|_2^2 - 2\delta_t \alpha_t \left\langle x_t - \hat{x}_t,\nabla f_{\mu}(x_t) - \nabla f_{\mu}(\hat{x}_t)\right\rangle + 4(n^2+2n)\alpha_t^2 L_{f,0}^2\\
&\quad\leq \ \delta_t^2\|x_t - \hat{x}_t\|_2^2 + 2\delta_t \alpha_t \rho\|x_t - \hat{x}_t\|_2^2 + 4(n^2+2n)\alpha_t^2 L_{f,0}^2\\
&\quad= \ \left( 1 - \left( 2\alpha_t (\bar{\rho}-\rho) + \alpha_t^2\bar{\rho}(2\rho-\bar{\rho})  \right)   \right)\|x_t - \hat{x}_t\|_2^2 + 4(n^2+2n)\alpha_t^2 L_{f,0}^2 ,
\end{split}
\end{equation*}
\noindent where the first equality follows from Lemma \ref{lemma: x_hat prox representation w.r.t. r}, the first inequality follows from non-expansiveness of the proximal operator (e.g. see \cite[Theorem 12.12]{Springer:RockWets}), the second inequality follows from the triangle inequality and \eqref{eqn: bound on the expected gradient}, while the third inequality follows from weak convexity of $f_{\mu}$ (see Proposition \ref{Proposition: weak convexity properties}). Since $\bar{\rho} \leq 2\rho$, the result follows.
\end{proof}
\par We can now establish the convergence rate of Algorithm \ref{Algorithm: Z-ProxSG}, in terms of the magnitude of the gradient of the Moreau envelope of the surrogate problem's objective function, that is $\phi_{\mu}^{1/\bar{\rho}}$.
\begin{theorem} \label{theorem: convergence general case}
Let Assumption \textnormal{\ref{Assumption: final assumption}} hold. Let also $\{x_t\}_{t=0}^T$ be the sequence of iterates produced by Algorithm \textnormal{\ref{Algorithm: Z-ProxSG}}, with $x_{t^*}$ being the point that the algorithm returns. For any $t\geq 0$, $\mu > 0$, and for any $\bar{\rho} \in (\rho,2\rho]$, it holds that
\begin{equation} \label{eqn: reduction in terms of Moreau envelope}
\begin{split}
\mathbb{E}_{U,\xi} \left[\phi_{\mu}^{1/\bar{\rho}}(x_{t+1}) \right] \leq & \ \mathbb{E}_{U,\xi} \left[\phi_{\mu}^{1/\bar{\rho}}(x_{t}) \right] - \frac{\alpha_t(\bar{\rho}-\rho)}{\bar{\rho}} \mathbb{E}_{U,\xi}\left[\left\| \nabla \phi_{\mu}^{1/\bar{\rho}}(x_t) \right\|_2^2 \right]\\
&\quad  + 2(n^2 + 2n)\bar{\rho}\alpha_t^2L_{f,0}^2,
\end{split}
\end{equation}
\noindent and $x_{t^*}$ satisfies
\begin{equation} \label{eqn: returned point Moreau gradient magnitude}
\mathbb{E}_{U,\xi}\left[\left\| \nabla \phi_{\mu}^{1/\bar{\rho}}(x_{t^*})\right\|_2^2 \right] \leq \frac{\bar{\rho}}{\bar{\rho}-\rho} \frac{\left(\phi_{\mu}^{1/\bar{\rho}}(x_0) - \underset{x}{\min}\ \phi_{\mu}(x) \right)	+ 2(n^2 + 2n)\bar{\rho}L_{f,0}^2 \sum_{t=0}^{T} \alpha_t^2}{\sum_{t=0}^T \alpha_t}.
\end{equation}
\noindent In particular, letting $\bar{\rho} = 2\rho$, $\Delta \geq \phi_{\mu}^{1/\bar{\rho}}(x_0) - \underset{x}{\min}\ \phi_{\mu}(x)$, and setting 
\begin{equation} \label{eqn: constant step-size bound}
\alpha_t = \frac{1}{2}\min\left\{\frac{1}{\rho}, \sqrt{\frac{\Delta}{(n^2+2n)\rho L_{f,0}^2(T+1)}}\right\}, 
\end{equation}
\noindent in  Algorithm \textnormal{\ref{Algorithm: Z-ProxSG}}, yields:
\begin{equation} \label{eqn: rate of convergence of Moreau gradient mangitude}
\mathbb{E}_{U,\xi}\left[\left\| \nabla \phi_{\mu}^{1/(2\rho)}(x_{t^*})\right\|_2^2 \right] \leq 8 \max\left\{ \frac{\Delta \rho}{T+1}, L_{f,0}\sqrt{\frac{\Delta \rho n(n+2)}{T+1}}\right\}.
\end{equation}
\end{theorem}
\begin{proof}
Using the definition of the Moreau envelope, we have
\begin{equation*}
\begin{split}
\scalemath{0.9}{\mathbb{E}_{U,\xi}^{t} \left[\phi_{\mu}^{1/\bar{\rho}}(x_{t+1})\right]} \leq &\ \mathbb{E}_{U,\xi}^{t} \left[\phi_{\mu}(\hat{x}_t) + \frac{\bar{\rho}}{2}\|\hat{x}_t-x_{t+1}\|_2^2 \right] \\
\leq &\ \scalemath{0.96}{\phi_{\mu}(\hat{x}_t) + \frac{\bar{\rho}}{2}\left[ \|x_t - \hat{x}_t\|_2^2 + 4(n^2+2n)\alpha_t^2 L_{f,0}^2 - 2\alpha_t (\bar{\rho}-\rho)\|x_t - \hat{x}_t\|_2^2\right]}\\
=&\ \phi_{\mu}^{1/\bar{\rho}}(x_t) + \bar{\rho}\left[ 2(n^2+2n)\alpha_t^2 L_{f,0}^2 - \alpha_t (\bar{\rho}-\rho)\|x_t - \hat{x}_t\|_2^2\right],
\end{split}
\end{equation*}
\noindent where the second inequality follows from Lemma \ref{lemma: descent property of iterates}, and the equality  follows from the definition of $\hat{x}_t$. Then, \eqref{eqn: reduction in terms of Moreau envelope} is derived by taking the expectation with respect to the filtration (all the data observed so far, i.e. $U_{t-1},\xi_{t-1},\ldots, U_0,\xi_0$). Inequality \eqref{eqn: returned point Moreau gradient magnitude} can be obtained as in \cite[Section 3]{SIAMOpt:Davis}, by rearranging and utilizing the closed form of the gradient of the associated Moreau envelope. 
\par Finally, by setting $\alpha_t$ as in \eqref{eqn: constant step-size bound}, separating cases, and plugging the respective expressions in \eqref{eqn: returned point Moreau gradient magnitude}, yields \eqref{eqn: rate of convergence of Moreau gradient mangitude} and completes the proof.
\end{proof}
\par The previous theorem provides an $\mathcal{O}\left(n^2 \epsilon^{-4}\right)$ convergence rate of Algorithm \ref{Algorithm: Z-ProxSG} for finding an $\epsilon$-stationary point of the Moreau envelope corresponding to \eqref{primal surrogate problem}, i.e. $\phi_{\mu}^{1/(2\rho)}$ (or, equivalently, an $\mathcal{O}(\sqrt{n}/T^{1/4})$ rate). In what follows, we would like to assess the quality of such a solution for the original problem \eqref{primal problem}. To that end, we will utilize Lemma \ref{lemma:Gaussian smoothing properties}. Before we proceed, let us provide certain well--known properties of the Moreau envelope, which indicate that it serves as a measure of closeness to optimality. We can observe (see \cite[Section 2.2]{SIAMOpt:Davis}) that for any $x \in \mathbb{R}^n$, and $\hat{x} \coloneqq \textbf{prox}_{\lambda \phi_{\mu}}(x)$, the following hold:
\[ \|\hat{x} - x\|_2 = \lambda \left\|\nabla \phi_{\mu}^{\lambda}(x)\right\|_2, \quad \phi_{\mu}\left( \hat{x}\right) \leq \phi_{\mu}(x), \quad \textnormal{dist}\left(0; \partial \phi_{\mu}(\hat{x})\right) \leq \left\|\nabla \phi_{\mu}^{\lambda}(x)\right\|_2,\]
\noindent where, given any closed set $\mathcal{A} \subset \mathbb{R}^n$, $\textnormal{dist}\left(z; \mathcal{A}\right) \coloneqq \inf_{z' \in \mathcal{A}}\|z-z'\|_2.$ In other words, a near-stationary point of $\phi_{\mu}^{1/(2\rho)}$ is close to a near-stationary point of $\phi_{\mu}$. We expect that if $\mathbb{E}_{U,\xi}\left[\left\|\nabla \phi_{\mu}^{1/\bar{\rho}}(x_{t^*})\right\|_2 \right]\leq \epsilon$, for some small $\epsilon >0$, then there will exist a small $\delta(\epsilon) > 0$ such that $\mathbb{E}_{U,\xi}\left[\textnormal{dist}\left(0,\partial\phi_{\mu}(x_{t^*})\right)\right] \leq \delta(\epsilon)$. Indeed, this is a standard assumption used in the literature (e.g. see \cite{SIAMOpt:Davis,COAP:KungurtsevRinaldi,Arxiv:Kozaketal}). The direct relation between $\delta$ and $\epsilon$ is not known in general, but in some cases this can be measured.  For example, if $\partial \phi_{\mu}$ is a sub-Lipschitz continuous mapping (see \cite[Definition 9.27]{Springer:RockWets}) or if $r$ is an indicator function to a compact convex set (see \cite{SIAMOPT:KalogeriasPowellZerothOrder}), then we obtain that $\delta = \mathcal{O}(\epsilon)$. 
\par In what follows, assuming that $\mathbb{E}_{U,\xi}\left[\textnormal{dist}\left(0,\partial\phi_{\mu}(x_{t^*})\right)\right] \leq \delta$, for some small $\delta > 0$, we show that $\mathbb{E}_{U,\xi}\left[\left\|\nabla \phi^{1/\bar{\rho}}(x_{t^*})\right\|^2_2 \right]\leq  \mathcal{O}\left(\delta^2 + \sqrt{n}\mu\right)$. To that end, in the following lemma we relate the Moreau envelope of the original problem's objective function $\phi^{\lambda}$ to the surrogate $\phi_{\mu}$ in \eqref{primal surrogate problem}.
\begin{lemma} \label{lemma: connection of Moreau envelope and surrogate function}
Let Assumption \textnormal{\ref{Assumption: final assumption}} hold. Given any $x \in \mathbb{R}^n$, any $\bar{\rho} \in (\rho,2\rho]$, and any $\mu > 0$, we have that
\[\left\langle x - \tilde{x},  v_{\mu}\right\rangle  \geq \frac{\bar{\rho}-\rho}{\bar{\rho}^2}\left\|\nabla \phi^{1/\bar{\rho}}(x)\right\|_2^2 - 2\mu L_{f,0} n^{\frac{1}{2}},\]
\noindent where $\tilde{x} \coloneqq \textnormal{\textbf{prox}}_{\bar{\rho}^{-1}\phi}(x)$, $\phi^{1/\bar{\rho}}$ is the Moreau envelope of $\phi$ in \textnormal{\eqref{primal problem}}, and $v_{\mu} \in  \partial \phi_{\mu}(x)$. 
\end{lemma}
\begin{proof}
\par See Appendix \ref{Appendix: proof of lemma connecting Moreau envelope and surrogate function}.
\end{proof}
\begin{theorem} \label{theorem: bound on Moreau envelope gradient of original problem}
Let Assumption \textnormal{\ref{Assumption: final assumption}} hold. Let $x_{\delta}$ be any $\delta$-stationary point of problem \textnormal{\eqref{primal surrogate problem}}, that is, there exists $v_{\mu} \in \partial \phi_{\mu}(x_{\delta})$, such that $\|v_{\mu}\|_2 \leq \delta$ (equivalently, $\textnormal{dist}\left(0,\partial \phi_{\mu}(x_{\delta})\right) \leq \delta$). Given any $\bar{\rho} \in (\rho,2\rho]$, and any $\mu > 0$, we have that $\left\lvert\phi\left(x_{\delta}\right) - \phi_{\mu}\left(x_{\delta}\right)\right\rvert \leq \mu L_{f,0} n^{\frac{1}{2}}$. Moreover,
\[\left\|\nabla \phi^{1/\bar{\rho}}(x_{\delta})\right\|_2^2 \leq \frac{\bar{\rho}^2}{\bar{\rho}-\rho} \left(\frac{\delta^2}{\bar{\rho}-\rho} + 4 \mu L_{f,0} n^{\frac{1}{2}}\right). \]
\noindent In particular, assuming that $\mathbb{E}_{U,\xi}\left[\textnormal{dist}\left(0,\partial\phi_{\mu}(x_{t^*})\right)\right] \leq \delta$, where $x_{t^*}$ is returned by Algorithm \ref{Algorithm: Z-ProxSG}, we obtain that
\[\mathbb{E}_{U,\xi}\left[\left\|\nabla \phi^{1/\bar{\rho}}(x_{t^*})\right\|_2^2\right] \leq \frac{\bar{\rho}^2}{\bar{\rho}-\rho} \left(\frac{\delta^2}{\bar{\rho}-\rho} + 4 \mu L_{f,0} n^{\frac{1}{2}}\right). \]

%\left(\frac{\epsilon \bar{\rho}}{\bar{\rho}-\rho}+ \sqrt{\left(\frac{\epsilon %\bar{\rho}}{\bar{\rho}-\rho}\right)^2 + \frac{8\bar{\rho}^2 \mu L_{f,0} %n^{\frac{1}{2}}}{\bar{\rho}-\rho}} \right). \]
\end{theorem}
\begin{proof}
\par The first part of the lemma follows immediately from the definition of $\phi_{\mu}$ and Lemma \ref{lemma:Gaussian smoothing properties}.
\par From Lemma \ref{lemma: connection of Moreau envelope and surrogate function}, we have that
\begin{equation} \label{eqn: monotonicity inequality}
\left\langle x_{\delta} - \tilde{x}_{\delta},  v_{\mu}\right\rangle  \geq \frac{\bar{\rho}-\rho}{\bar{\rho}^2}\left\|\nabla \phi^{1/\bar{\rho}}(x_{\delta})\right\|_2^2 - 2\mu L_{f,0} n^{\frac{1}{2}},
\end{equation}
\noindent where $\tilde{x}_{\delta} \coloneqq \textbf{prox}_{\bar{\rho}^{-1} \phi}\left(x_{\delta}\right)$. From the triangle inequality, we obtain
\begin{equation*}\left\|\nabla \phi^{1/\bar{\rho}}(x_{\delta})\right\|_2^2-\frac{\delta \bar{\rho}}{\bar{\rho}-\rho}  \left\|\nabla \phi^{1/\bar{\rho}}(x_{\delta})\right\|_2 - \frac{2\bar{\rho}^2 \mu L_{f,0} n^{\frac{1}{2}} }{\bar{\rho}-\rho}  \leq 0,
\end{equation*}
\noindent where we used the definition of $\tilde{x}_{\delta}$, the expression of the gradient of $\phi^{1/\bar{\rho}}(x_{\delta})$, and the assumption that $\|v_{\mu}\|_2 \leq \delta$. For ease of presentation, we introduce some notation. Let $u \coloneqq \left\|\nabla \phi^{1/\bar{\rho}}(x_{\delta})\right\|_2$,  $\beta \coloneqq -\frac{\delta \bar{\rho}}{\bar{\rho}-\rho}$, and $\gamma \coloneqq -\frac{2\bar{\rho}^2 \mu L_{f,0} n^{\frac{1}{2}} }{\bar{\rho}-\rho}$. We proceed by finding an upper bound for $u$, so that the previous inequality is satisfied. This is trivial, since we can equate this inequality to zero, and find the most-positive solution of the quadratic equation in $u$. Indeed, it is easy to see that 
\[ u \leq \frac{1}{2}\left(-\beta + \sqrt{\beta^2 - 4\gamma}\right).\]
\noindent Thus we easily obtain $u^2 \leq \left(\beta^2 - 2\gamma\right)$. The first bound then follows immediately by plugging the values of $\beta$ and $\gamma$.
\par Finally, by assuming that $\mathbb{E}_{U,\xi}\left[\textnormal{dist}\left(0,\partial\phi_{\mu}(x_{t^*})\right)\right] \leq \delta$, substituting $x_{t^*}$ in \eqref{eqn: monotonicity inequality}, taking total expectations and repeating the previous analysis, yields the second bound and completes the proof.
\end{proof}

\begin{remark}
Let us notice that the convergence rate in Theorem \textnormal{\ref{theorem: convergence general case}} is given in terms of the expected squared gradient norm of the surrogate Moreau envelope evaluated at the output of Algorithm \textnormal{\ref{Algorithm: Z-ProxSG}}, that is $\mathbb{E}_{U,\xi}\left[\left\|\nabla \phi^{1/\bar{\rho}}_{\mu}(x_{t^*})\right\|^2_2\right]$. This is in line with the results presented in \textnormal{\cite{COAP:KungurtsevRinaldi}}, however, the authors of the aforementioned paper did not investigate the error introduced by considering the surrogate problem. In this paper, we attempted to do this in Theorem \textnormal{\ref{theorem: bound on Moreau envelope gradient of original problem}}. Ideally, we would like to provide a rate on $\mathbb{E}_{U,\xi}\left[\left\|\nabla \phi^{1/\bar{\rho}}(x_{t^*})\right\|_2^2\right]$. In the special cases where $r$ is an indicator function to a compact convex set or $\partial \phi$ is a sub-Lipschitz mapping, this can be done easily (e.g. see \textnormal{\cite[Section 6.4.2]{SIAMOPT:KalogeriasPowellZerothOrder}}). In the general case, and without additional restrictive assumption  (as in \cite{Arxiv:Nazarietal}), we are able to show that any near-stationary point for the surrogate problem is near-stationary for the Moreau envelope of the original function, with the approximation improving for smaller values of $\mu$. Thus, assuming that $x_{t^*}$ is near-stationary in expectation for the surrogate problem \textnormal{\eqref{primal surrogate problem}}, we were able to show that it will be near-stationary in expectation for the Moreau envelope corresponding to \textnormal{\eqref{primal problem}}. 
\end{remark}
\section{Numerical results} \label{sec: numerical results}
\par In this section we provide numerical evidence for the effectiveness of the proposed approach. Firstly, we run the method on certain phase retrieval instances taken from \cite{SIAMOpt:Davis} and compare the proposed zeroth-order approach, outlined in Algorithm \ref{Algorithm: Z-ProxSG}, against the double smoothing zeroth-order proximal stochastic gradient method analyzed in \cite{COAP:KungurtsevRinaldi}, a uniform smoothing zeroth-order method (e.g. see \cite{Arxiv:Nazarietal}), the simultaneous perturbation stochastic approximation method (originally proposed in \cite{IEETAC:Spall}), as well as the stochastic sub-gradient method proposed and analyzed in \cite{SIAMOpt:Davis}, noting that the latter method is significantly more difficult to employ (and implement) in the general case, since it assumes knowledge of sub-gradient information. In order to obtain a meaningful comparison, all zeroth-order schemes are using a constant step-size and constant smoothing parameter. For completeness, the four algorithms used in our comparison are outlined in Algorithm \ref{Algorithm: TPZ-ProxSG}, \ref{Algorithm: UniZ-ProxSG}, \ref{Algorithm: SPSA}, and \ref{Algorithm: ProxSSG},  respectively. Next, we verify that the proposed approach performs almost identically to the method outlined in \cite{COAP:KungurtsevRinaldi}, while being easier to tune and analyze (and additionally requiring $n$ less flops per iteration). 
\par Subsequently, we employ the proposed algorithm for the important task of tuning the parameters of optimization algorithms in order to obtain good and consistent behaviour for a wide range of optimization problems. We note that this problem can only be tackled by zeroth-order schemes, since there is no availability of first-order information. In particular, we employ a proximal alternating direction method of multipliers (pADMM) for the solution of PDE-constrained optimization instances. It is well-known that the behaviour of ADMM is heavily affected by the choice of its penalty parameter, and thus, we employ Algorithm \ref{Algorithm: Z-ProxSG} in order to find a nearly optimal value (in a sense to be described) for this parameter that allows the method to behave well for similar (out-of-sample) PDE-constrained optimization instances. To our knowledge, the heuristic model proposed for achieving this task is novel and highly effective.
\par The code is written in MATLAB and can be found on GitHub \footnote{\url{https://github.com/spougkakiotis/Z-ProxSG}}. The experiments were run on MATLAB 2019a, on a PC with a 2.2GHz Intel core i7 processor (hexa-core), 16GM RAM, using the Windows 10 operating system.
\renewcommand{\thealgorithm}{DSZ-ProxSG}

\begin{algorithm}
\caption{Double Smoothing Z-ProxSG}
    \label{Algorithm: TPZ-ProxSG}
\begin{algorithmic}
\State \textbf{Input:}  $x_0 \in \textnormal{dom}(r)$, a sequence $\{\alpha_t\}_{t \geq 0} \subset \mathbb{R}_+$, $\mu_1 \geq 2\mu_2 > 0$, and $T > 0$.
\For {($t = 0,1,2,\ldots, T$)}
\State Sample $\xi_t$, $U_{t,1},\ U_{t,2} \sim \mathcal{N}\left(0_n,I_n\right)$, and set
\[x_{t+1} = \textbf{prox}_{\alpha_t r}\left(x_t - \alpha_t G\left(x_t,U_{t,1},U_{t,2},\xi_t\right)\right), \]
\State where 
\[G\left(x_t,U_{t,1},U_{t,2},\xi_t\right) = \mu_2^{-1}\left(F\left(x_t + \mu_1 U_{t,1} + \mu_2 U_{t,2},\xi_t\right) - F(x_t + \mu_1 U_{t,1},\xi_t)\right) U_{t,2}.\]
\EndFor
\end{algorithmic}
\end{algorithm}
\renewcommand{\thealgorithm}{UniZ-ProxSG}

\begin{algorithm}
\caption{Uniform Z-ProxSG}
    \label{Algorithm: UniZ-ProxSG}
\begin{algorithmic}
\State \textbf{Input:}  $x_0 \in \textnormal{dom}(r) \subset \mathbb{R}^d$, a sequence $\{\alpha_t\}_{t \geq 0} \subset \mathbb{R}_+$, $\mu > 0$, and $T > 0$.
\For {($t = 0,1,2,\ldots, T$)}
\State Sample $\xi_t$, and $U_{t}$ uniformly from the $d$-dimensional ball, and set
\[x_{t+1} = \textbf{prox}_{\alpha_t r}\left(x_t - \alpha_t G\left(x_t,U_{t},\xi_t\right)\right), \]
\State where 
\[G\left(x_t,U_{t},\xi_t\right) = \frac{d}{\mu}\left(F\left(x_t,\xi_t\right) - F(x_t + \mu U_{t},\xi_t)\right) U_{t}.\]
\EndFor
\end{algorithmic}
\end{algorithm}

\renewcommand{\thealgorithm}{SPSA}

\begin{algorithm}
\caption{Simultaneous Perturbation Stochastic Approximation}
    \label{Algorithm: SPSA}
\begin{algorithmic}
\State \textbf{Input:}  $x_0 \in \textnormal{dom}(r)$, a sequence $\{\alpha_t\}_{t \geq 0} \subset \mathbb{R}_+$, $\mu_1 \geq 2\mu_2 > 0$, and $T > 0$.
\For {($t = 0,1,2,\ldots, T$)}
\State Sample $\xi_t$, and $U_{t}$ from a $d$-dimensional Bernoulli distribution, and set
\[x_{t+1} = \textbf{prox}_{\alpha_t r}\left(x_t - \alpha_t G\left(x_t,U_{t},\xi_t\right)\right), \]
\State with 
\[G\left(x_t,U_{t},\xi_t\right) = \frac{F\left(x_t + \mu U_{t},\xi_t\right) - F(x_t - \mu U_{t},\xi_t)}{2\mu U_t},\]
\State where the division is component-wise.
\EndFor
\end{algorithmic}
\end{algorithm}
\renewcommand{\thealgorithm}{ProxSSG}

\begin{algorithm}[!ht]
\caption{Proximal Stochastic Sub-Gradient}
    \label{Algorithm: ProxSSG}

\begin{algorithmic}
\State \textbf{Input:}  $x_0 \in \textnormal{dom}(r)$, a sequence $\{\alpha_t\}_{t \geq 0} \subset \mathbb{R}_+$, and $T > 0$.
\For {($t = 0,1,2,\ldots, T$)}
\State Sample $\xi_t$, and set
\[x_{t+1} = \textbf{prox}_{\alpha_t r}\left(x_t - \alpha_t G\left(x_t,\xi_t\right)\right), \]
\State where $G\left(x_t,\xi_t\right) \in \partial F(x_t,\xi_t)$.
\EndFor
\end{algorithmic}
\end{algorithm}

\subsection{Phase retrieval} \label{subsec: Phase retrieval}
\par Let us first focus on the solution of phase retrieval problems. Following \cite{SIAMOpt:Davis}, we generate standard Gaussian measurements $a_i \sim \mathcal{N}(0,I_d)$ for $i = 1,\ldots,m$, a target signal $\bar{x}$ as well as a starting point $x_0$ on the unit sphere. Then, by setting $b_i = \langle a_i,\bar{x} \rangle^2$, for $i = 1,\ldots,m$, we want to solve 
\[ \min_{x \in \mathbb{R}^d} f(x) = \frac{1}{m} \sum_{i=1}^m \left\lvert \langle a_i,x\rangle^2 - b_i\right\rvert.\]
\par As discussed in \cite{SIAMOpt:Davis}, this is a weakly convex optimization problem. We attempt to solve it using Algorithms \ref{Algorithm: Z-ProxSG}, \ref{Algorithm: TPZ-ProxSG}, \ref{Algorithm: UniZ-ProxSG}, \ref{Algorithm: SPSA}, and \ref{Algorithm: ProxSSG}. For this specific instance, we can explicitly compute the sub-gradient appearing in Algorithm \ref{Algorithm: ProxSSG}. Specifically, as shown in \cite[Section 5.1]{SIAMOpt:Davis}, the subdifferential of the function $f_i(x) \coloneqq \lvert \langle a_i,x\rangle^2 - b_i \rvert$ reads
\[\partial f_i(x) = 2\langle a_i,x\rangle \cdot \begin{cases} 
      \textnormal{sign}\left(\langle a_i,x\rangle^2 - b_i \right), & \textnormal{if}\ \langle a_i, x\rangle \neq 0, \\
      [-1,1], & \textnormal{otherwise} 
   \end{cases}. \]
 \noindent At each iteration of Algorithm \ref{Algorithm: ProxSSG} we choose the sub-gradient that yields the highest objective value reduction.
\par Before proceeding with the experiments, let us discuss some implementation details. Each of the tested algorithms is heavily affected by the choice of the step-size $\alpha_t$. We choose this parameter to be constant. For Algorithms \ref{Algorithm: Z-ProxSG}, \ref{Algorithm: TPZ-ProxSG}, \ref{Algorithm: UniZ-ProxSG}, and \ref{Algorithm: SPSA}, by loosely following the theory in Section \ref{sec: prox stoch sub-gradient method}, we set it to $\alpha_t = \frac{1}{2d\sqrt{T}}$ for all $t \geq 0$. Similarly, for Algorithm \ref{Algorithm: ProxSSG}, following \cite[Section 3]{SIAMOpt:Davis}, we set $\alpha_t = \frac{1}{2\sqrt{T}}$. Finally, Algorithms \ref{Algorithm: Z-ProxSG}, \ref{Algorithm: UniZ-ProxSG}, and \ref{Algorithm: SPSA} are quite robust with respect to the choice of the smoothing parameter $\mu$ (or $\mu_1,\ \mu_2$, for Algorithm \ref{Algorithm: TPZ-ProxSG}). For Algorithms \ref{Algorithm: Z-ProxSG}, \ref{Algorithm: UniZ-ProxSG}, and \ref{Algorithm: SPSA} this was set to $\mu = 5\cdot 10^{-10}$. From Theorem \ref{theorem: bound on Moreau envelope gradient of original problem} we observe that the smaller the value of $\mu$, the better the quality of the obtained solution (in terms of closeness to a stationary point of the Moreau envelope of the objective function). Indeed, there is no ``optimal" value for $\mu$ and hence we set it to an as small as possible value, considering numerical accuracy issues that can arise due to finite machine precision. For Algorithm \ref{Algorithm: TPZ-ProxSG}, by loosely following the theory in \cite[Section 2.2]{IEEE_Inf_Th:Duchi_etal}, we set $\mu_1 = 5\cdot 10^{-7},\ \mu_2 = 5\cdot 10^{-10}$. Notice that we enforce $\mu = \mu_2$ in order to observe a comparable numerical behaviour between all zeroth-order schemes.
\par We set up 6 optimization problems, with varying sizes $(d,m)$. In every case, the maximum number of iterations is set as $T = 2\cdot 10^3 \cdot m$. The random seed of MATLAB was set to \emph{"shuffle"}, which is initiated based on the current time. For each pair of sizes we produce 15 instances and run each of the five methods for $T$ iterations. In Figure \ref{Fig: Phase Retrieval Conv Profiles}, we present the average convergence profiles with 95\% confidence intervals for each method. 

\begin{figure}[h!]
    \centering
    \includegraphics[height = 0.33\textwidth,width=0.49\textwidth]{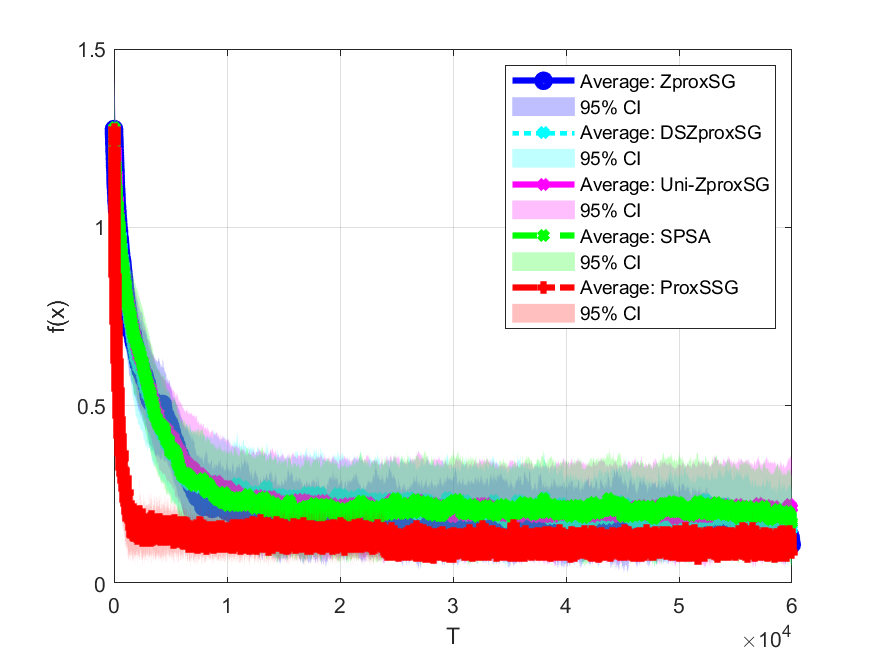}\hfill
    \includegraphics[height = 0.33\textwidth,width=0.49\textwidth]{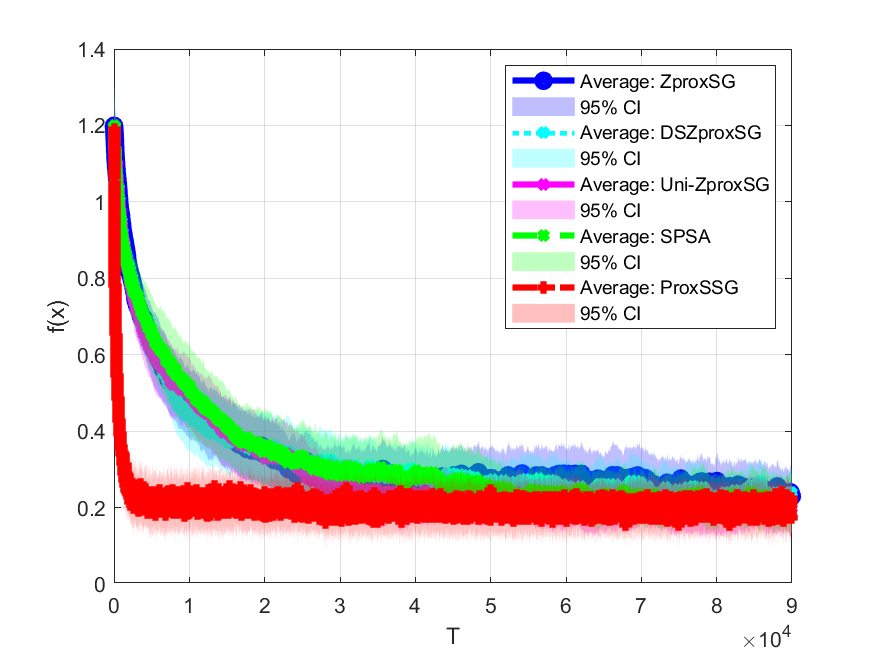}\hfill\\[1mm]
    \includegraphics[height = 0.33\textwidth,width=0.49\textwidth]{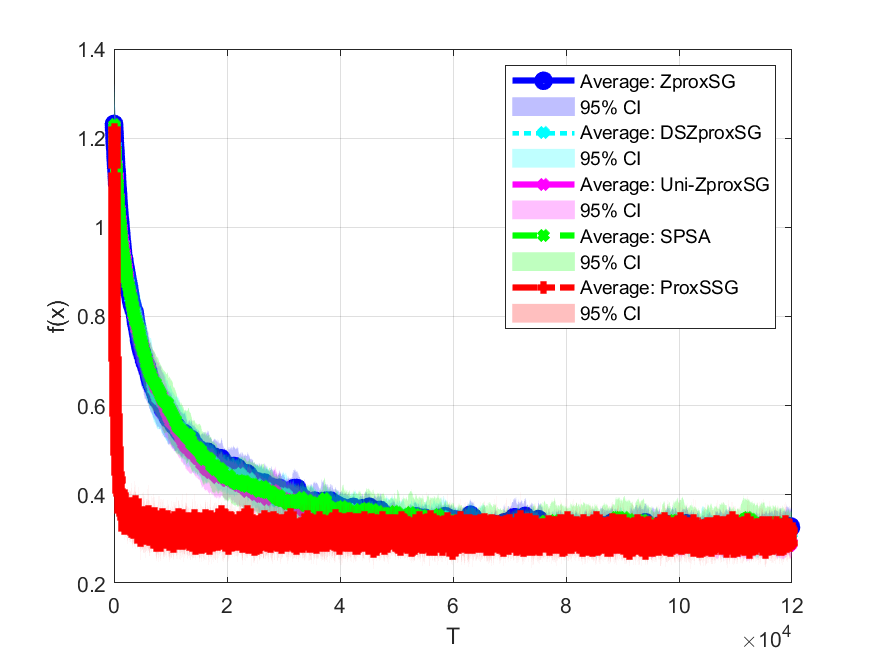}\hfill
     \includegraphics[height = 0.33\textwidth,width=0.49\textwidth]{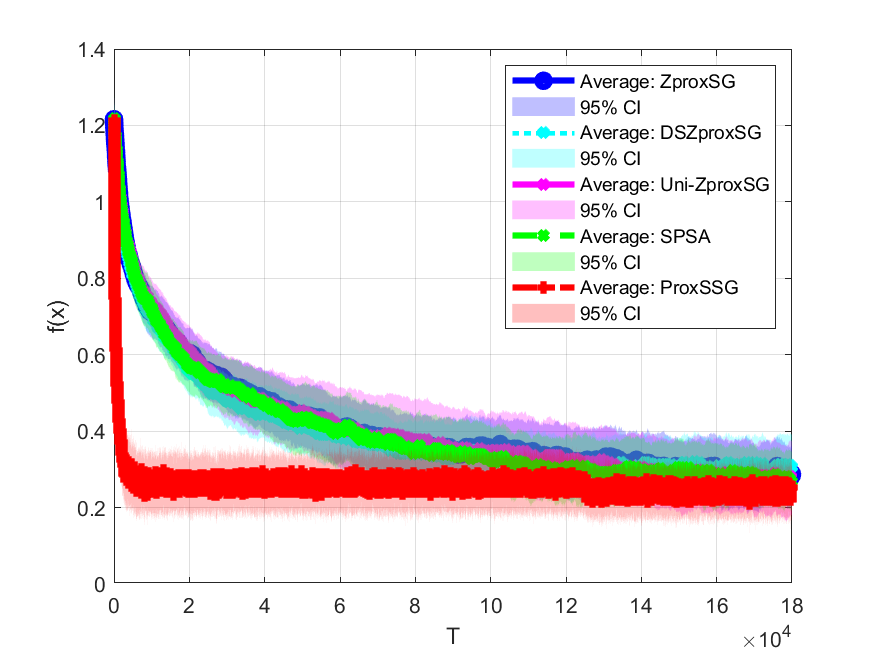}\hfill\\[1mm]
    \includegraphics[height = 0.33\textwidth,width=0.49\textwidth]{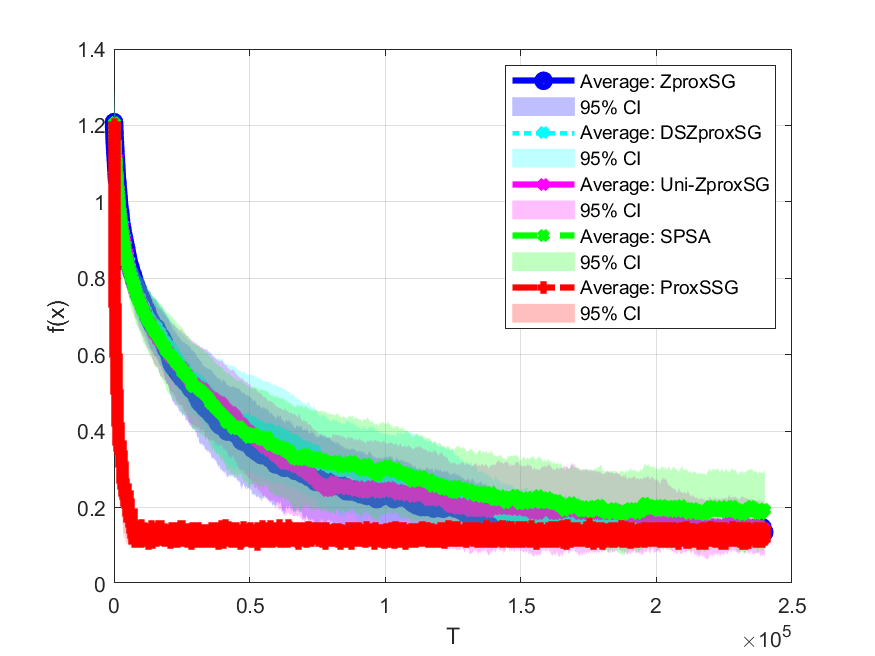}\hfill
     \includegraphics[height = 0.33\textwidth,width=0.49\textwidth]{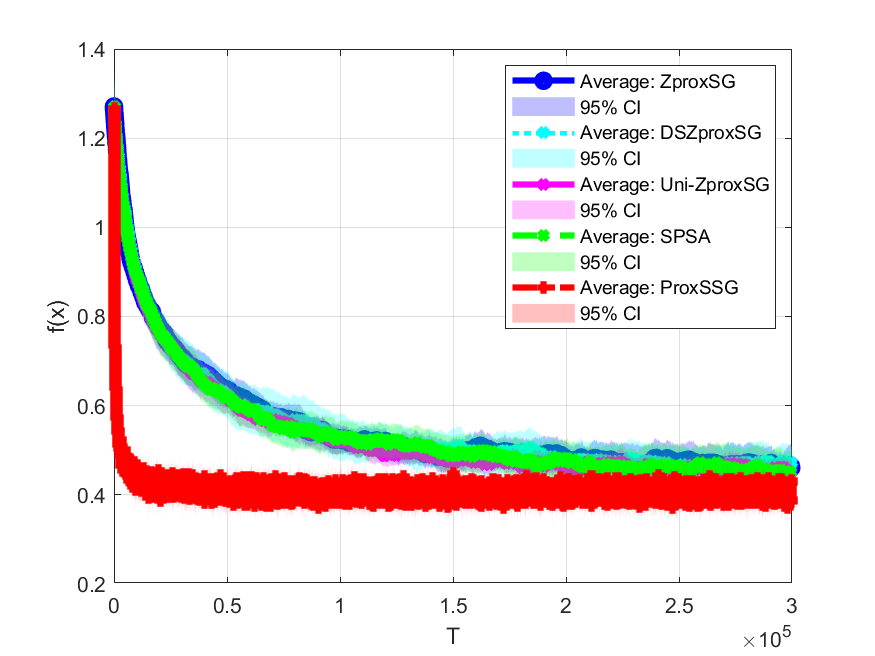}\hfill\\[1mm]
    \caption{Convergence profiles for Z-ProxSG, DSZ-ProxSG, Uni-ZproxSG, SPSA and ProxSSG: average objective function value (lines) and 95\% confidence intervals (shaded regions) vs number of iterations. The upper row corresponds, from left to right, to $(d,m) = (10,30),\ (20,45)$. The middle row corresponds, from left to right, to $(d,m) = (40,60),\ (35,90)$. The lower row corresponds, from left to right, to $(d,m) = (30,120),\ (80,150)$.\label{Fig: Phase Retrieval Conv Profiles}}
\end{figure}    
\par We can draw several useful observations from Figure \ref{Fig: Phase Retrieval Conv Profiles}. Firstly, while the convergence of the zeroth-order schemes is slower, as compared to the convergence of the sub-gradient scheme (as we expected from the theory), the obtained solutions are comparable for all algorithms. On the other hand, all zeroth-order schemes have a very similar behaviour, which was expected as we used similar values for the smoothing parameters. Let us notice that the theory in Section \ref{sec: conv analysis} can easily be altered to apply for Algorithm \ref{Algorithm: UniZ-ProxSG}, since the Gaussian and the uniform smoothing techniques are very similar (see, for example, the analysis in \cite{IEEE_Inf_Th:Duchi_etal}). Algorithm \ref{Algorithm: SPSA} seems to behave equally well, compared to the other zeroth-order schemes, however, no convergence analysis is available in the literature for problems of the form of \eqref{primal problem}. Standard convergence analyses for SPSA are available for (stochastic) convex programming instances, allowing adaptive choices for the step-size $\alpha_t$ as well as the smoothing parameter $\mu$. However, the adaptive choices proposed in \cite{IEEETAES:Spall} for convex programming did not deliver convergence for the phase retrieval instances solved here, thus we tuned this algorithm in the same way we tuned all the other zeroth-order schemes. In order to verify that Algorithms \ref{Algorithm: Z-ProxSG} and \ref{Algorithm: TPZ-ProxSG} behave seemingly identically even if we tune the ratio $\mu_1/\mu_2$, we set $(d,m) = (40,60)$ and run the two zeroth-order methods using various values of $(\mu_1,\mu_2)$, always ensuring that $\mu = \mu_2$. The results, which are averaged over 15 randomly generated instances, are reported in Figure \ref{Fig: Phase Retrieval Conv Profiles mu comparison}.
\begin{figure}[h!]
    \centering
    \includegraphics[width=0.31\textwidth]{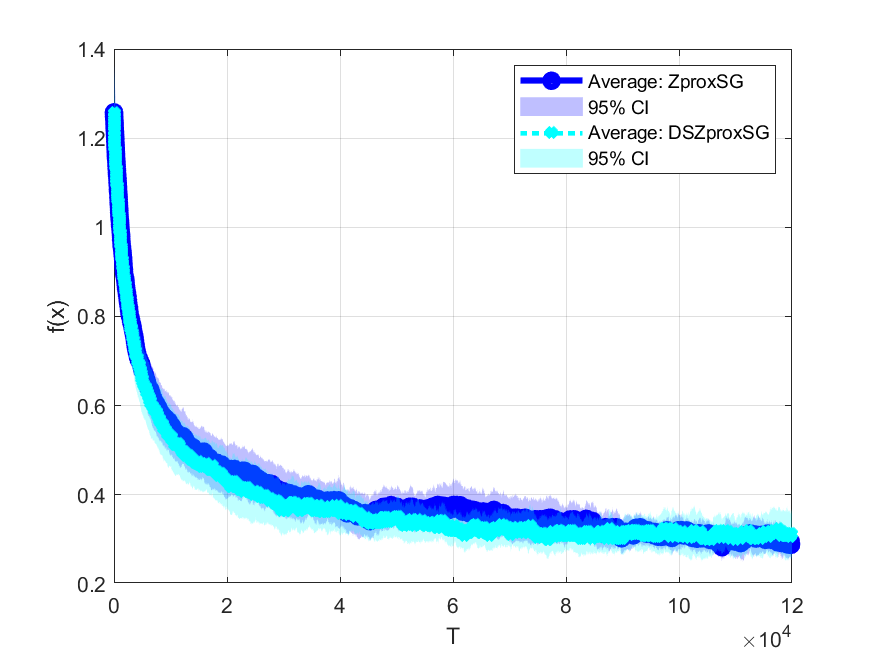}\hfill
    \includegraphics[width=0.31\textwidth]{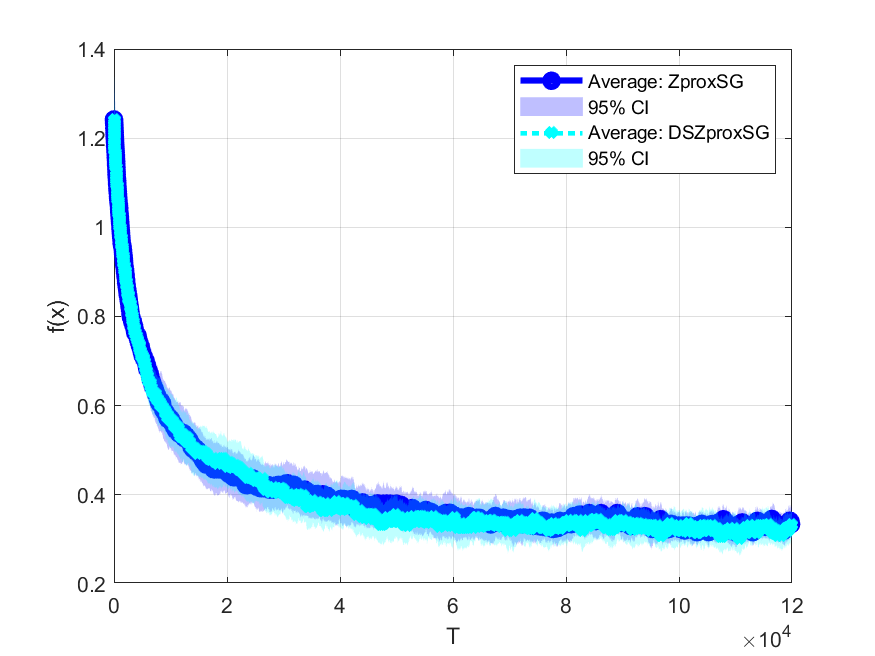}\hfill
    \includegraphics[width=0.31\textwidth]{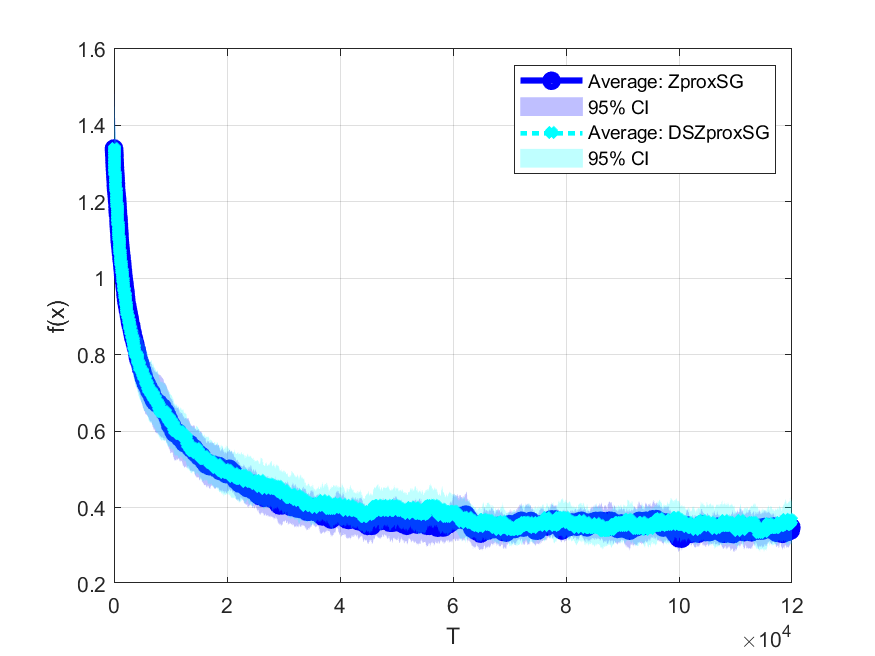}\\[1mm]
    \includegraphics[width=0.31\textwidth]{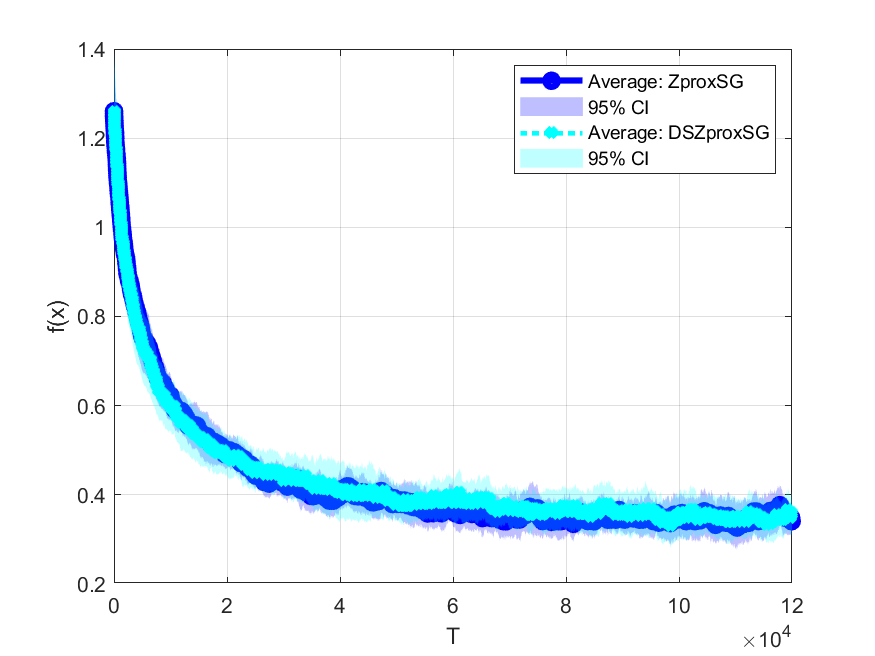}\hfill
    \includegraphics[width=0.31\textwidth]{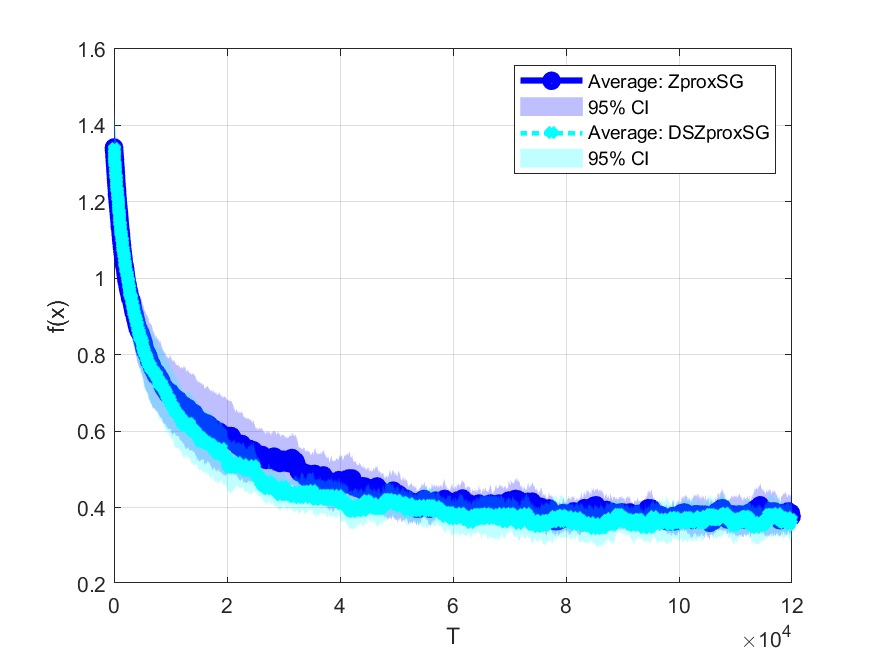}\hfill
     \includegraphics[width=0.31\textwidth]{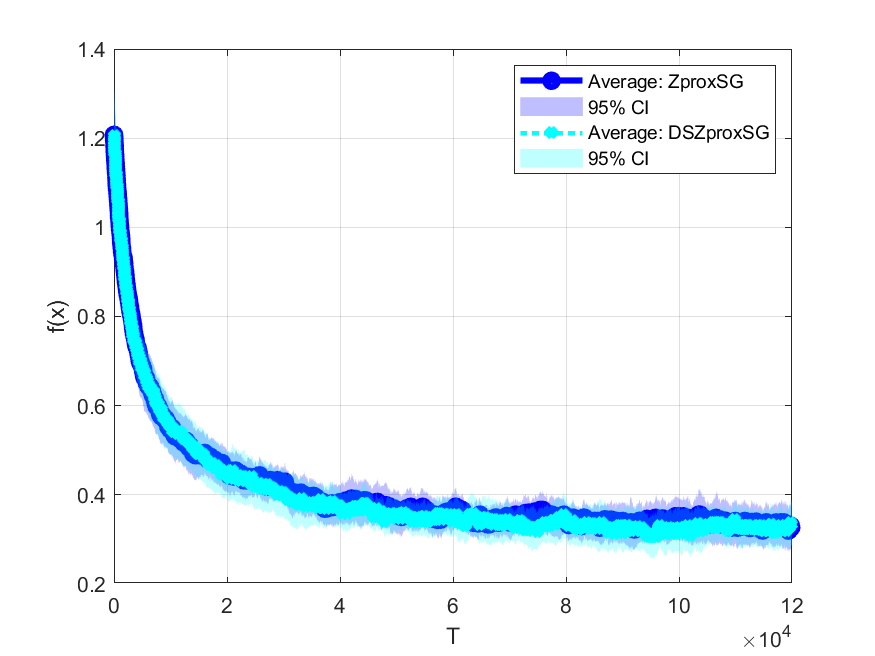}\hfill\\[1mm]
    \caption{Convergence profiles for Z-ProxSG, DSZ-ProxSG: average objective function value (lines) and 95\% confidence intervals (shaded regions) vs number of iterations, for $(d,m) = (40,60)$. The upper row corresponds, from left to right, to $(\mu_1,\mu_2) = (10^{-x},10^{-y})$, $x = 4, 5, 6$, $y = 7$. The lower row corresponds, from left to right, to $(\mu_1,\mu_2) = (10^{-x},10^{-y})$, $x = 6, 7, 8$, $y = 9$. In each case we set $\mu = \mu_2$.\label{Fig: Phase Retrieval Conv Profiles mu comparison}}
\end{figure}    
\par We note that the authors in \cite{IEEE_Inf_Th:Duchi_etal} show that for convex programming instances a proper tuning of the ratio $\mu_1/\mu_2$ can lead to a better convergence rate for the double-smoothing as compared to the single smoothing, in terms of its dependence on the dimension of the problem (noting that this has not been shown for weakly convex problems of the form of \eqref{primal problem} in \cite{COAP:KungurtsevRinaldi}). As we observe in Figure \ref{Fig: Phase Retrieval Conv Profiles mu comparison}, varying this ratio does not seem to have any actual effect in practice, since we observe that for a wide range of values for $\mu_1/\mu_2$ the double-Gaussian smoothing method behaves seemingly identically.
\par Notice that we could obtain better results by extensively tuning $\alpha_t$ and $T$ for each instance, however, we provided general values that seem to exhibit a very consistent behaviour for all of the presented schemes.
\subsection{Hyper-parameter tuning for optimization methods}
\par Next, we consider the problem of tuning hyper-parameters of optimization algorithms, so as to improve their robustness and efficiency over a chosen set of optimization instances. The discussion in this section will be restricted to the case of an alternating direction method of multipliers (see \cite{BoydADMM} for an introductory review of ADMMs), although we conjecture that the same technique can be employed for tuning a much wider range of optimization methods.
\subsubsection{Proximal ADMM for PDE-constrained optimization}
\par In this section, we are interested in the solution of optimization problems with partial differential equation (PDE) constraints via a proximal alternating direction method of multipliers (pADMM). We note that various other applications would be suitable for the presented method, however, we restrict the problem pool for ease of presentation. 
\par We consider optimal control problems of the following form:
\begin{equation} \label{generic inverse problem}
\begin{split}
\min_{\mathrm{y},\mathrm{u}} \     &\ \mathrm{J}\left(\mathrm{y}(\bm{x}),\mathrm{u}(\bm{x})\right), \\
\text{s.t.}\ &\ \mathrm{D} \mathrm{y}(\bm{x}) - \mathrm{u}(\bm{x}) = \mathrm{g}(\bm{x}),\\
       &\ \mathrm{u_{a}}(\bm{x}) \leq \mathrm{u}(\bm{x})  \leq \mathrm{u_{b}}(\bm{x}),
\end{split}
\end{equation}
\noindent where $(\rm{y},\rm{u}) \in \mathcal{H}_1(K) \times \mathcal{L}_2(K)$,  $\mathrm{J}\left(\mathrm{y}(\bm{x}),\mathrm{u}(\bm{x})\right)$ is a convex functional defined as
\begin{equation} \label{generic inverse problems objective function}
\begin{split}
\mathrm{J}\left(\mathrm{y}(\bm{x}),\mathrm{u}(\bm{x})\right) \coloneqq &\ \frac{1}{2}\| \rm{y} - \bar{\rm{y}}\|_{\mathcal{L}_2(K)}^2 + \frac{\beta_1}{2}\|\rm{u}\|_{\mathcal{L}_1(K)}^2 + \frac{\beta_2}{2}\|\rm{u}\|_{\mathcal{L}_2(K)}^2,
\end{split}
\end{equation}
\noindent  $\mathrm{D}$ denotes a linear differential operator, $\bm{x}$ is a $2$-dimensional spatial variable, and $\beta_1,\ \beta_2 \geq 0$ denote the regularization parameters of the control variable.
\par The problem is considered on a given compact spatial domain $K\subset \mathbb{R}^{2}$ with boundary $\partial K$, and is equipped with Dirichlet boundary conditions. The algebraic inequality constraints are assumed to hold a.e. on $K$. We further note that ${\rm u_a}$ and ${\rm u_b}$ are chosen as constants, although a more general formulation would be possible.  In what follows, we consider two classes of state equations (i.e. the equality constraints in \eqref{generic inverse problem}): the Poisson's equation, as well as the convection--diffusion equation. For the Poisson optimal control, by following \cite{NLAA:PearsonPorcStoll}, we set the desired state as $\bar{\mathrm{y}} = \sin(\pi x_1)\sin(\pi x_2)$. For the convection-diffusion, which reads as $-\epsilon \rm{\Delta y} + \rm{w} \cdot \nabla y = u$, where $\rm{w}$ is the wind vector given by $\rm{w} = [2x_2(1-x_1)^2, -2x_1(1-x_2^2)]^\top$, we set the desired state as $\rm{\bar{y}} = exp(-64((x_1 - 0.5)^2 + (x_2 - 0.5)^2))$ with zero boundary conditions (e.g. see \cite[Section 5.2]{NLAA:PearsonPorcStoll}). The diffusion coefficient $\epsilon$ is set as $\epsilon = 0.05$. In both cases, we set $K = (0,1)^2$, $\rm{u_a} = -2$, and $\rm{u_b} = 1.5$ (see \cite{NLAA:PearsonPorcStoll}).
\par We solve problem \eqref{generic inverse problem} via a \emph{discretize-then-optimize} strategy. We employ the Q1 finite element discretization implemented in IFISS\footnote{\url{https://personalpages.manchester.ac.uk/staff/david.silvester/ifiss/default.htm}} (see \cite{IFISSACM,IFISSSIAMREVIEW}). This yields a sequence of $\ell_1$-regularized convex quadratic programming problems of the following form:
\begin{equation} \label{discretized PDE problem} 
\underset{x \in \mathbb{R}^n}{\text{min}} \  c^\top x + \frac{1}{2} x^\top Q x + \|Dx\|_1 + \delta_{\mathcal{K}}(x), \qquad \text{s.t.}  \  Ax = b,
\end{equation}
\noindent where $A \in \mathbb{R}^{m\times n}$ models the linear constraints, $D \in \mathbb{R}^{n\times n}$ is a diagonal matrix, and $\mathcal{K}$ models the restrictions on the discretized control variables. We note that the discretization of the smooth part of the objective of problem \eqref{generic inverse problem} follows a standarad Galekrin approach (e.g. see \cite{AMS:Trolzsch}), while the $\mathcal{L}_1$ term is discretized by the \emph{nodal quadrature rule} as in \cite{COAP:SongChenYu,ESAIM:GerdDaniel} (which achieves a first-order convergence--see \cite{ESAIM:GerdDaniel}).
\par We can reformulate problem \eqref{discretized PDE problem} by introducing an auxiliary variable $w \in \mathbb{R}^n$, as follows
\begin{equation} \label{discretized PDE ADMM reformulation} 
\underset{x \in \mathbb{R}^n, w \in \mathbb{R}^n}{\text{min}} \  c^\top x + \frac{1}{2} x^\top Q x + \|Dw\|_1 + \delta_{\mathcal{K}}(w), \qquad \text{s.t.}  \  Ax = b, \quad w - x = 0.
\end{equation}
\par Given a penalty $\sigma > 0$, we associate the following augmented Lagrangian to \eqref{discretized PDE ADMM reformulation}
\begin{equation*}
\begin{split}
{L}_{\sigma}(x,w,y_1, y_2) \coloneqq &\ c^\top x + \frac{1}{2} x^\top Q x + g(w) + \delta_{\mathcal{K}}(w) - y_1^\top (Ax-b) -y_2^\top(w-x) \\&\ + \frac{\sigma}{2}\|Ax-b\|^2 + \frac{\sigma}{2}\|w-x\|^2.
\end{split}
\end{equation*}
\noindent Let an arbitrary positive definite matrix $R_x$ be given, and assume the notation $\|x\|^2_{R_x} = x^\top R_x x$. We now provide (in Algorithm \ref{proximal ADMM algorithm}) a proximal ADMM for the approximate solution of \eqref{discretized PDE ADMM reformulation}.
\renewcommand{\thealgorithm}{pADMM}
\begin{algorithm}[!ht]
\caption{Proximal Alternating Direction Method of Multipliers}
    \label{proximal ADMM algorithm}
    \textbf{Input:}  $\sigma > 0$, $R_x \succ 0$, $\gamma \in \left(0,\frac{1+\sqrt{5}}{2}\right)$, $(x_0,w_0,y_{1,0},y_{2,0}) \in \mathbb{R}^{3 n + m}$.
\begin{algorithmic}
\For {($t = 0,1,2,\ldots$)}
\State $w_{t+1} = \underset{w}{\arg\min}\left\{{L}_{\sigma}\left(x_t,w,y_{1,t},y_{2,t}\right) \right\} \equiv \Pi_{\mathcal{K}}\left(\textbf{prox}_{\sigma^{-1} g}\left(x_t + \sigma^{-1} y_{2,t}\right)\right).$
\State $x_{t+1} = \underset{x}{\arg\min}\left\{{L}_{\sigma}\left(x,w_{t+1},y_{1,t},y_{2,t}\right) + \frac{1}{2}\|x-x_t\|_{R_x}^2\right\}.$
\State $y_{1,t+1} = y_{1,t} - \gamma\sigma(Ax_{t+1}-b)$.
\State $y_{2,t+1} = y_{2,t} - \gamma\sigma(w_{t+1}-x_{t+1})$.
\EndFor
\end{algorithmic}
\end{algorithm}
\par We notice that under feasibility and convexity assumptions on \eqref{discretized PDE ADMM reformulation}, Algorithm \ref{proximal ADMM algorithm} is able to achieve global convergence potentially at a linear rate, assuming strong convexity (see \cite{SciComp:DengYin}), even in cases where $R_x$ is not positive definite \cite{JIMO:JiangWuCai}. Here we assume that $R_x$ is positive definite, and we employ it as a means of reducing the memory requirements of Algorithm \ref{proximal ADMM algorithm}. More specifically, given some constant $\hat{\sigma} > 0$, such that $\hat{\sigma}I_n - \textnormal{Off}(Q) \succ 0$, we define 
\[R_x = \hat{\sigma}I_n - \textnormal{Off}(Q),\]
\noindent where $\textnormal{Off}(B)$ denotes the matrix with zero diagonal and off-diagonal elements equal to the off-diagonal elements of $B$. We note that this method was employed in \cite{Arxiv:PougkGond} as a means of obtaining a starting point for a semi-smooth Newton-proximal method of multipliers, suitable for the solution of \eqref{discretized PDE problem}. 
\par In the experiments to follow, Algorithm \ref{proximal ADMM algorithm} uses the zero vector as a starting point, while the step-size is set to the value $\gamma = 1.618$. The penalty parameter $\sigma$ is given to the algorithm by the user, and this is later utilized to tune the method over an appropriate set of problem instances. We expect that different values for $\sigma$ should be chosen when considering Poisson and convection-diffusion problems. Thus, in the following subsection we tune Algorithm \ref{proximal ADMM algorithm} for each of the two problem-classes separately.
\subsubsection{Automated tuning: problem formulation and numerical results}
\par Given a positive number $k$, we consider a general stochastic optimization problem of the following form
\begin{equation}  \label{automated tuning problem}
\min_{\sigma \in \mathbb{R}} f(\sigma;k) \coloneqq \mathbb{E}\left[F(\sigma,\xi;k)\right] + \delta_{\left[ \sigma_{\min},\sigma_{\max}\right]}\left(\sigma\right), \qquad \xi \sim P,
\end{equation}
\noindent where $f(\sigma;k) = $``expected residual reduction of Algorithm \ref{proximal ADMM algorithm} after $k$ iterations, given the penalty parameter $\sigma$, for discretized problems of the form of \eqref{discretized PDE problem} originating from a distribution $P$". We assume that $\xi \in \Xi \subset \mathbb{R}^d$, where a sample $\xi$ is a specific problem instance of the form of \eqref{discretized PDE problem}. In particular, we consider two different tuning problems, and thus two different distributions $P_1,\ P_2$. Sampling either of the two distributions $P_1,\ P_2$ yields a problem of the form of \eqref{discretized PDE problem} with arbitrary (but sensible) values for the regularization parameters $\beta_1,\ \beta_2 > 0$, as well as a randomly chosen (grid-based) problem size. For $P_1$, the linear constraints model the Poisson equation, while for $P_2$ the convection-diffusion equation. The values for the remaining problem parameters (i.e. control bounds, desired states, wind vector, and diffusion coefficient) are given in the previous subsection.
\begin{remark}
Notice that the choice of $f(\cdot;k)$ in \textnormal{\eqref{automated tuning problem}} has multiple motivations. Firstly, by choosing a small value for $k$ (e.g. 10 or 15), we can ensure that each run of Algorithm \textnormal{\ref{proximal ADMM algorithm}} will not take excessive time (since one run of the algorithm corresponds to a sample-function evaluation within Algorithm \textnormal{\ref{Algorithm: Z-ProxSG}}). Additionally, the scale of $f(\cdot;k)$ is expected to be comparable for very different classes of problems. Indeed, assuming that Algorithm \textnormal{\ref{proximal ADMM algorithm}} does not diverge (which could only happen if an infeasible instance was tackled), we expect that in most cases $0 \leq f(\cdot; k) \leq C$, where $C = \mathcal{O}(1)$ is a small positive value, irrespectively of the problem under consideration, since we measure the residual reduction. However, it should be noted that this is a heuristic. Indeed, finding the parameter value that yields the fastest residual reduction in the first $k$ iterations does not necessarily yield an optimal convergence behaviour in the long-run. Nonetheless, we can always increase the value of $k$ at the expense of a more expensive meta-tuning. In both cases considered here, this was not required.
\par Finally, we note that the constraints in \textnormal{\eqref{automated tuning problem}} arise from prior information that we might have about the class of problems that we consider. It is well-known that very small or very large values for the penalty parameter of the ADMM tend to perform poorly (e.g. see the discussions in \textnormal{\cite[Section 3.4.1.]{BoydADMM}} or \textnormal{\cite{Teixetal}}). Thus, some limited preliminary experimentation can determine suitable values for $\sigma_{\min}$ and $\sigma_{\max}$ for each problem class that is considered. In the experiments to follow we set $\sigma_{\min} = 10^{-2}$ and $\sigma_{\max} = 10^2$.
\end{remark}
\par In order to find an approximate solution to \eqref{automated tuning problem}, we need to define a representative discrete training set from the space of optimization problems produced by $P_1$ (or $P_2$, respectively). To that end, we will use a discrete training set $\hat{\Xi} = \{\xi_1,\ldots,\xi_m\} \subset \Xi$, which yields the following problem
\begin{equation}  \label{approximate automated tuning problem}
\min_{\sigma \in \mathbb{R}} f(\sigma;k) \coloneqq \frac{1}{m}\sum_{j = 1}^m F(\sigma,\xi_j;k) + \delta_{\left[ \sigma_{\min},\sigma_{\max}\right]}\left(\sigma\right).
\end{equation}
\noindent Once an approximate solution to \eqref{approximate automated tuning problem} is found, we can test its quality on out-of-sample PDE-constrained optimization instances. For both problem classes (i.e. Poisson and convection-diffusion optimal control), we construct 80 optimization instances. In particular, we define the sets 
\begin{equation*}
\begin{split}
\mathcal{B}_1 \coloneqq \{0, 10^{-2}, 10^{-4}, 10^{-6}\},\ \mathcal{B}_2 \coloneqq \{0, 10^{-2}, 10^{-4}, 10^{-6}\}, \\ \mathcal{M} \coloneqq \{(2^3+1)^2, (2^4+1)^2, (2^5+1)^2, (2^6+1)^2, (2^7+1)^2\}, 
\end{split}
\end{equation*} 
\noindent where $\mathcal{B}_1$ ($\mathcal{B}_2$, respectively) contains potential values for $\beta_1$ ($\beta_2$, respectively), while $\mathcal{M}$ contains potential problem sizes. At each iteration $t$ of Algorithm \ref{Algorithm: Z-ProxSG}, we sample uniformly $\beta_{t,1} \in \mathcal{B}_1,\ \beta_{t,2} \in \mathcal{B}_2$, and $n_t \in \mathcal{M}$, and use the triple $\xi = (\beta_{t,1},\beta_{t,2},n_t)$ to generate an optimization instance. Then, $F(\cdot,\xi;k)$ can be evaluated by running Algorithm \ref{proximal ADMM algorithm} on this instance for $k$ iterations and subsequently computing the residual reduction. In the following runs of Algorithm \ref{Algorithm: Z-ProxSG}, we set $\mu = 5\cdot 10^{-10}$, and $T = 200\cdot m$, where $m = \lvert \mathcal{B}_1\rvert \cdot \lvert \mathcal{B}_2 \rvert \cdot \lvert \mathcal{M} \rvert = 80$.
\paragraph{Poisson optimal control} Let us first consider Poisson optimal control problems. We apply Algorithm \ref{Algorithm: Z-ProxSG} to find an approximate solution of \eqref{approximate automated tuning problem}, with $k = 15$.  We choose $\sigma^*$ as the last iteration of Algorithm \ref{Algorithm: Z-ProxSG}, which in this case turned out to be $\sigma^* = 0.2778$. Then, in order to evaluate the quality of this penalty, we run Algorithm \ref{proximal ADMM algorithm} on 40 randomly-chosen out-of-sample Poisson optimal control problems for different penalty values $\sigma \in [\sigma_{\min},\sigma_{\max}]$, including $\sigma^*$. In particular, in order to create out-of-sample instances, we define the sets 
\begin{equation*}
\begin{split}
\hat{\mathcal{B}}_1 \coloneqq \{10^{-3}, 5\cdot 10^{-3}, 10^{-5}, 5\cdot 10^{-5}\},\ \hat{\mathcal{B}}_2 \coloneqq \{10^{-3}, 5\cdot 10^{-3}, 10^{-5}, 5\cdot 10^{-5}\}, \\ \hat{\mathcal{M}} \coloneqq \{(2^3+1)^2, (2^4+1)^2, (2^5+1)^2, (2^6+1)^2, (2^7+1)^2,(2^8+1)^2\}, 
\end{split}
\end{equation*}   
\noindent These correspond to 96 optimization instances, that were not used during the zeroth-order meta-tuning. The averaged convergence profiles (measuring the scaled residual versus the ADMM iteration) are summarized in Figure \ref{Fig: Poisson out-of-sample automated}.

\begin{figure}[h!]
    \centering
    \includegraphics[width=0.8\textwidth]{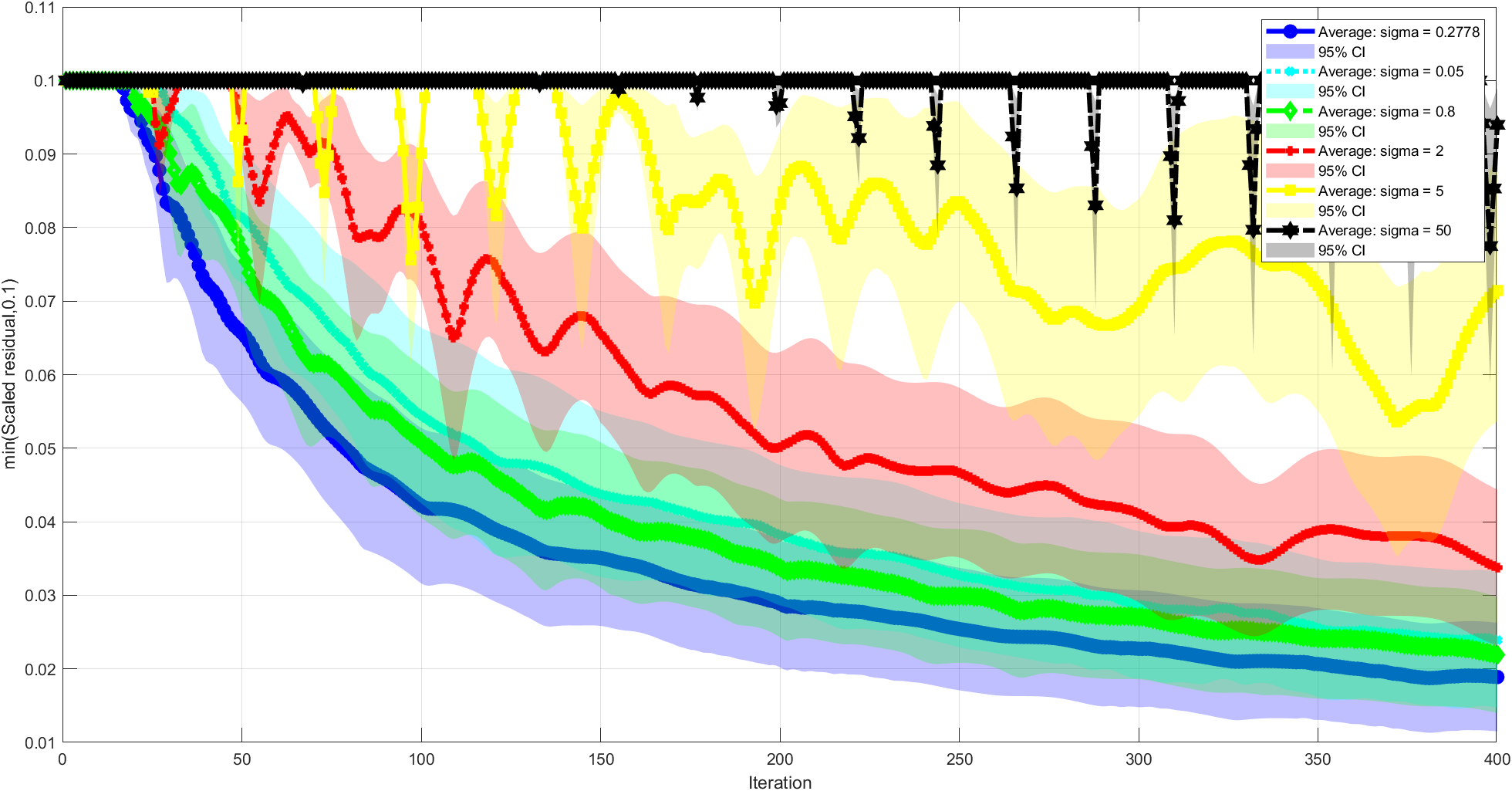}
    \caption{Convergence profiles for pADMM with varying penalty parameter $\sigma$: average residual reduction (lines) and 95\% confidence intervals (shaded regions) vs number of pADMM iterations. The algorithm is run over 40 randomly selected (out-of-sample) Poisson optimal control problems.\label{Fig: Poisson out-of-sample automated}}
\end{figure}    

\par In Figure \ref{Fig: Poisson out-of-sample automated} we observe that out of the 6 different values for $\sigma$, Algorithm \ref{proximal ADMM algorithm} exhibits the most consistent behaviour when using the value that Algorithm \ref{Algorithm: Z-ProxSG} suggested as ``optimal". The next two best-performing values were $\sigma = 0.8$, $\sigma = 0.05$, and one can observe these are the ones closest to $\sigma^* = 0.2778$. Let us notice that the $y-$axis in Figure \ref{Fig: Poisson out-of-sample automated} only shows values less than $0.1$. This was enforced for readability purposes.

\paragraph{Optimal control of the convection-diffusion equation} We now consider the optimal control of the convection-diffusion equation. As before, we apply Algorithm \ref{Algorithm: Z-ProxSG} to find an approximate solution of \eqref{approximate automated tuning problem}, with $k = 15$.  We choose $\sigma^*$ as the last iteration of Algorithm \ref{Algorithm: Z-ProxSG}, which in this case turned out to be $\sigma^* = 5.7004$. We evaluate the quality of this penalty by running Algorithm \ref{proximal ADMM algorithm} on 40 randomly-chosen out-of-sample convection-diffusion optimal control problems for different penalty values $\sigma \in [\sigma_{\min},\sigma_{\max}]$, including $\sigma^*$. As before these instances are created by sampling the previously defined sets $\hat{\mathcal{B}}_1$, $\hat{\mathcal{B}}_2$ and $\hat{\mathcal{M}}$. The averaged convergence profiles (measuring the scaled residual versus the ADMM iteration) are summarized in Figure \ref{Fig: Conv. Diff. out-of-sample automated}. 

\begin{figure}[h!]
    \centering
    \includegraphics[width=0.8\textwidth]{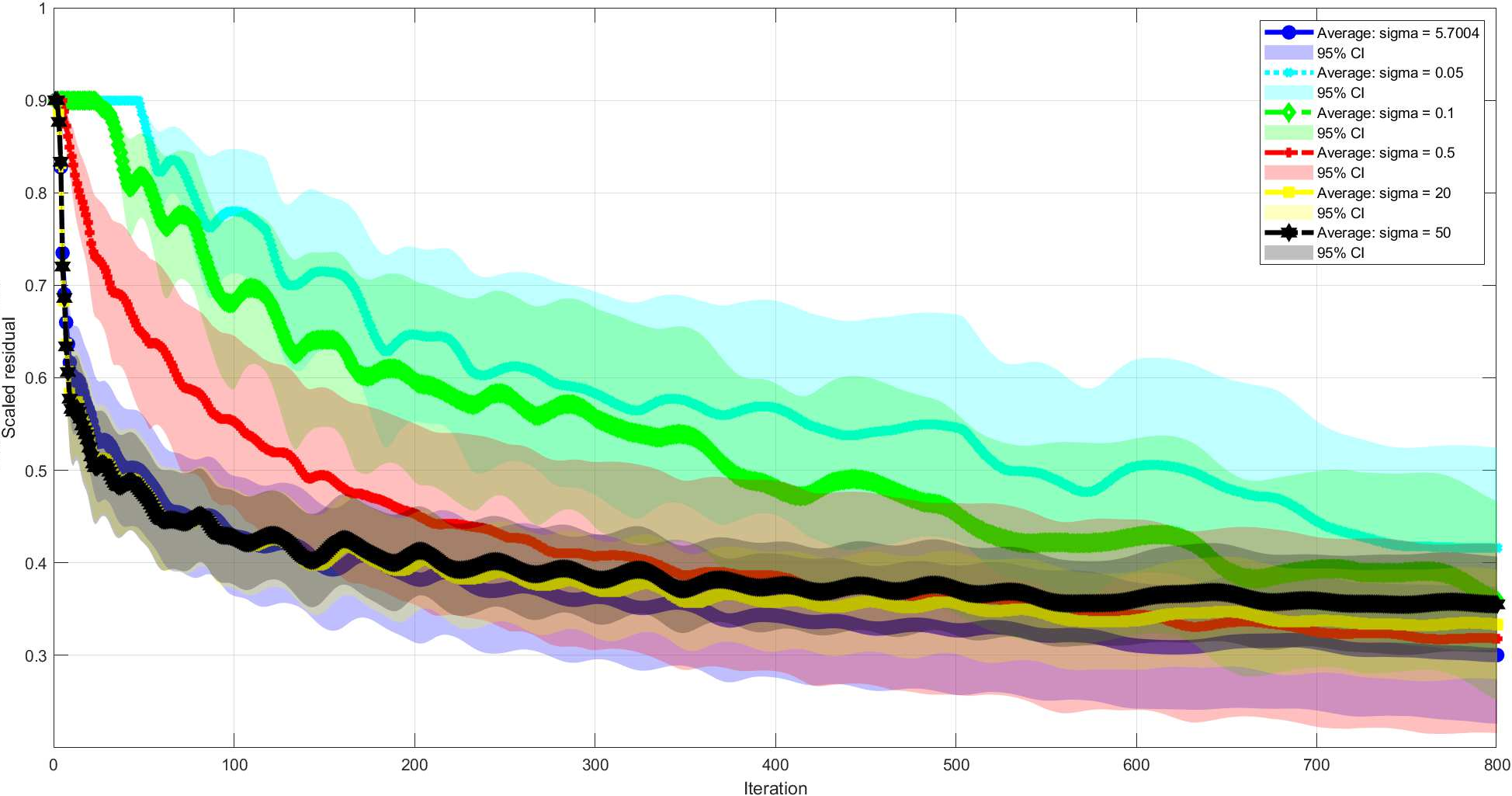}
    \caption{Convergence profiles for pADMM with varying penalty parameter $\sigma$: average residual reduction (lines) and 95\% confidence intervals (shaded regions) vs number of pADMM iterations. The algorithm is run over 40 randomly selected (out-of-sample) convection-diffusion optimal control problems.\label{Fig: Conv. Diff. out-of-sample automated}}
\end{figure}    
\par Based on the results shown in Figure \ref{Fig: Conv. Diff. out-of-sample automated} we can observe that Algorithm \ref{Algorithm: Z-ProxSG} is indeed able to find a value for $\sigma$ that approximately minimizes the residual reduction of the ADMM during the first $k$ iterations. However, as already noted, that this is not necessarily the optimal choice when running Algorithm \ref{proximal ADMM algorithm} for a much larger number of iterations. We expect that in many cases (e.g. as in the optimal control of the Poisson equation) the first few iterations of the ADMM are sufficient to predict the behaviour of the algorithm in later iterations. On the other hand, from the convection-diffusion instances we observe that a very steep residual reduction during the first ADMM iterations (e.g. observed when $\sigma = 50$ or $\sigma = 20$) does not necessarily result in the minimum achievable residual reduction after a large number of ADMM iterations. Of course this could be taken into account by increasing the value of $k$ (e.g. the users might set it equal to the number of iterations that they are willing to let ADMM run for the specific application at hand), but it should be noted that this would result in more expensive sample-function evaluations of problem \eqref{automated tuning problem}. Other heuristics could also improve the generalization performance of the model in \eqref{automated tuning problem} (such as employing different starting point strategies for the ADMM runs during the ``training"). However, the focus of this paper prevents us from investigating this matter any further. Most importantly, in both problem classes, we were able to observe that Algorithm \ref{Algorithm: Z-ProxSG} succeeds in finding an approximate solutions to \eqref{automated tuning problem}, yielding efficient versions of Algorithm \ref{proximal ADMM algorithm}.
\section{Conclusions} \label{sec: Conclusions}
\par In this paper we have derived and analyzed a zeroth-order proximal stochastic gradient method suitable for the solution of weakly convex stochastic optimization problems. We demonstrated that, under standard assumptions, the algorithm is guaranteed to converge to a near-stationary solution of the problem at a rate comparable to that achieved by similar sub-gradient schemes. The theoretical results were consistently verified numerically on certain phase-retrieval instances, supporting the viability of the proposed approach. Finally, we developed a novel heuristic model for the calculation of ``optimal" hyper-parameters of optimization algorithms for an arbitrary given class of problems. Using the latter, we were able to showcase that the proposed zeroth-order algorithm can be efficiently employed for hyper-parameter tuning problems, yielding very promising results.
\appendix
\section{Appendix}
\subsection{Proof of Lemma \ref{lemma: x_hat prox representation w.r.t. r}} \label{Appendix: proof of lemma x_hat prox representation w.r.t. r}
\begin{proof}
From the definition of $\hat{x}_t$ we have
\begin{equation*}
\begin{split}
\scalemath{0.94}{\alpha_t \bar{\rho}\left(x_t - \hat{x}_t\right) \in \alpha_t \partial r\left(\hat{x}_t\right) + \alpha_t \nabla f_{\mu}(\hat{x}_t) \Leftrightarrow }&\ \scalemath{0.94}{\alpha_t \bar{\rho} x_t - \alpha_t \nabla f_{\mu}(\hat{x}_t) + \delta_t \hat{x}_t \in \hat{x}_t + \alpha_t \partial r\left( \hat{x}_t \right)} \\
\scalemath{0.94}{\Leftrightarrow} &\ \scalemath{0.94}{\hat{x}_t = \textnormal{\textbf{prox}}_{\alpha_t r}\left( \alpha_t\bar{\rho}x_t - \alpha_t \nabla f_{\mu}(x_t) + \delta_t \hat{x}_t \right)}.
\end{split}
\end{equation*}
This completes the proof.
\end{proof}
\subsection{Proof of Lemma \ref{lemma: connection of Moreau envelope and surrogate function}} \label{Appendix: proof of lemma connecting Moreau envelope and surrogate function}
%\iffalse
\begin{proof}
\par Following \cite[Lemma 5.2]{SIAMOPT:KalogeriasPowellZerothOrder}, we begin by noticing that for any  $x_1,\ x_2 \in \mathbb{R}^n$ the following holds
\begin{equation*} 
\begin{split}
\phi(x_1) - \phi(x_2) = &\ \phi_{\mu}(x_1) + \phi(x_1) - \phi_{\mu}(x_1) - \phi_{\mu}(x_2) - \phi(x_2) + \phi_{\mu}(x_2) \\
\leq &\ \phi_{\mu}(x_1)  - \phi_{\mu}(x_2) +  2\sup_{x \in \mathbb{R}^n} \lvert \phi_{\mu}(x) - \phi(x) \rvert \\
\leq &\ \phi_{\mu}(x_1)  - \phi_{\mu}(x_2) + 2\mu L_{f,0}n^{\frac{1}{2}},
\end{split}
\end{equation*}
\noindent where the  second inequality follows from \eqref{eqn: approximation error of surrogate}. On the other hand, given $v_{\mu} \in \partial \phi_{\mu}(x_t)$, from $\rho$-weak convexity of $\phi_{\mu}(\cdot)$, and by utilizing Proposition \ref{Proposition: weak convexity properties}, we obtain
\begin{equation*}
\begin{split}
\langle x_1 - x_2, v_{\mu}\rangle \geq &\ \phi_{\mu}(x_1) - \phi_{\mu}(x_2) - \frac{\rho}{2}\|x_1 - x_2\|_2^2 \\
 \geq  &\ \phi(x_1) - \phi(x_2) - \frac{\rho}{2}\|x_1-x_2\|_2^2 - 2\mu L_{f,0} n^{\frac{1}{2}},
\end{split}
\end{equation*}
\noindent for any $x_1,\ x_2 \in \mathbb{R}^n$. By letting $x_1 = x$ and $x_2 = \tilde{x} \coloneqq \textbf{prox}_{\bar{\rho}^{-1}\phi}(x)$, and by noting that $\bar{\rho} > \rho$, we obtain
\begin{equation*}
\begin{split}
\langle x - \tilde{x}, v_{\mu}\rangle \geq &\ \phi(x) - \phi(\tilde{x}) - \frac{\rho}{2}\|x-\tilde{x}\|_2^2 - 2\mu L_{f,0} n^{\frac{1}{2}} \\
\equiv &\ \phi(x) + \frac{\bar{\rho}}{2} \|x - x\|_2^2 - \left(\phi(\tilde{x}) + \frac{\bar{\rho}}{2} \|\tilde{x} - x\|_2^2\right) \\
&\quad + \frac{\bar{\rho}-\rho}{2}\|\tilde{x} - {x}\|_2^2 - 2\mu L_{f,0}n^{\frac{1}{2}}
\end{split}
\end{equation*}
\noindent However, we know that the map $y \mapsto \left( \phi(y) + \frac{\bar{\rho}}{2}\|y - x\|_2^2\right)$ is strongly convex with parameter $\bar{\rho}-\rho$, and is minimized at $\tilde{x}$, and thus \[\phi(x) + \frac{\bar{\rho}}{2} \|x - x\|_2^2 - \left(\phi(\tilde{x}) + \frac{\bar{\rho}}{2} \|\tilde{x} - x\|_2^2\right) \geq  \frac{\bar{\rho}-\rho}{2}\|x - \tilde{x}\|_2^2.\]
\noindent Hence, we obtain
\begin{equation*}
\begin{split}
\langle x - \tilde{x}, v_{\mu}\rangle \geq &\ (\bar{\rho}-\rho)\|\tilde{x} - {x}\|_2^2 - 2\mu L_{f,0}n^{\frac{1}{2}}\\
\equiv &\ \frac{\bar{\rho}-\rho}{\bar{\rho}^2}\|\nabla \phi^{1/\bar{\rho}}(x)\|_2^2 - 2\mu L_{f,0}n^{\frac{1}{2}},
\end{split}
\end{equation*}
\noindent where the last equivalence follows from the characterization of the gradient of the Moreau envelope, as well as the definition of $\tilde{x}_t$, and completes the proof.
\end{proof}
\bibliographystyle{siamplain}
\bibliography{references}

\end{document}